\documentclass[11pt,a4paper]{article}
\usepackage[utf8]{inputenc}
\usepackage[english]{babel}
\usepackage{amsmath}
\usepackage{amsthm}
\usepackage{amsfonts}
\usepackage{amssymb}
\usepackage{array}  
\usepackage{graphicx} 
\usepackage{xcolor}
\usepackage{mathrsfs}
\usepackage{dsfont}
\usepackage{subcaption}
\usepackage{multicol}
\usepackage{lipsum}
\usepackage{geometry}
\usepackage{hyperref}

\hypersetup{
    linktoc=page,
    linkcolor=red,          
    citecolor=blue,        
    filecolor=blue,        
    urlcolor=cyan,
    colorlinks=true}

\theoremstyle{plain}
\newtheorem{theorem}{Theorem}[section]
\newtheorem{proposition}[theorem]{Proposition}
\newtheorem{lemma}[theorem]{Lemma}
\newtheorem{corollary}[theorem]{Corollary}

\theoremstyle{definition}

\newtheorem{remark}[theorem]{Remark}

\newcommand{\N}{\mathbb{N}}
\newcommand{\Z}{\mathbb{Z}}
\newcommand{\R}{\mathbb{R}}

\newcommand{\E}{\mathbb{E}}
\newcommand{\Pro}{\mathbb{P}}

\newcommand{\Cov}{\textup{Cov}}

\newcommand{\Proba}{\mathbf{P}}
\newcommand{\Ex}{\mathbf{E}}

\newcommand{\Qu}{\mathrm{Q}}

\newcommand{\COV}{\mathrm{COV}}

\newcommand{\defn}{\textit}

\usepackage{mathtools}
\mathtoolsset{showonlyrefs}

\author{Vincent Tassion\footnote{ETH Zürich, Rämistrasse 101, 8092 Zürich, Switzerland. Email address of VT: vincent.tassion@math.ethz.ch. Current address and email address of HV: Institut Fourier, 100 rue des mathématiques, 38402 Saint-Martin-d'Hères, France; hugo.vanneuville@univ-grenoble-alpes.fr} \and Hugo Vanneuville\footnotemark[1]}

\title{Noise sensitivity of percolation via differential inequalities}

\date{}

\begin{document}

\maketitle

\abstract{Consider critical Bernoulli percolation in the plane. We give a new proof of the sharp noise sensitivity theorem shown by Garban, Pete and Schramm in \cite{GPS10}. Contrary to the previous approaches, we do not use any spectral tool. We rather study differential inequalities satisfied by a dynamical four-arm event, in the spirit of Kesten's proof of scaling relations in \cite{Kes87}. We also obtain new results in dynamical percolation. In particular, we prove that the Hausdorff dimension of the set of times with both primal and dual percolation equals $2/3$ a.s.



\tableofcontents

\section{Introduction}\label{sec:intro}
\subsection{Noise sensitivity of percolation}\label{ssec:noise}

Given some integer $n \ge 1$, we consider the discrete hypercube $\{0,1\}^n$ endowed with the uniform probability measure
\[
\Proba = \left( \frac{\delta_1+\delta_0}{2} \right)^{\otimes n}
\]
and we let $\Ex$ be the corresponding expectation. Let $t \in [0,1]$, let $X \sim \Proba$, and let $Y\sim \Proba$ be obtained from $X$ by resampling independently each coordinate with probability $t$. A function $f$ from $\{0,1\}^n$ to $\{0,1\}$ is called a Boolean function, and we study the covariance
\[
\COV_t(f) = \Cov(f(X),f(Y)) = \E [ f(X) f(Y) ]-\mathbf E[f]^2.
\]
(As usually, $\Pro$ and $\E$ refer to the underlying probability measure and expectation under which $X$ and $Y$ are defined, hence our use of the two different notations $(\Ex,\Proba)$ and $(\E,\Pro)$.)

It is well known that $\COV_t(f)$ is non-increasing in $t$ (see for instance Proposition~\ref{prop:2}), and in this paper we are interested in the following question: For which values of $t$ does the quantity  $\COV_t(f)$ get small? Or, in words, when do $f(X)$ and $f(Y)$ start to become ``roughly independent''? Noise sensitivity, in the sense of \cite{BKS99}, corresponds to the asymptotic version of this question.

In the present study, it will be convenient to introduce the quadratic form $\Qu_t$, defined by 
\[
\Qu_t(f)=\E [ f(X) f(Y) ]
\]
for every $f:\{0,1\}^n\to\{0,1\}$, and $\Qu_t(A)=\Qu_t(\textbf{1}_A)$ for every event $A\subseteq\{0,1\}^n$.  With this notation, the covariance introduced above can be rewritten as $\COV_t(f)= \Qu_t(f) - \Qu_1(f)$. 
\medskip

In the present work, we consider Boolean functions in the context of critical face percolation on the (regular) hexagonal lattice (which is the same as critical site percolation on the triangular lattice). Although all planar percolation models ``with enough symmetries'' are believed to behave the same (asymptotically), much more is known about this model, due to the proof of conformal invariance by Smirnov \cite{Smi01}. By using this property, the SLE processes introduced by Schramm, and Kesten's proof of scaling relations \cite{Kes87}, Lawler, Schramm and Werner \cite{LSW02} and Smirnov and Werner \cite{SW01} have computed several \textit{critical exponents} (see Section~\ref{ssec:exp}). This will be crucial in Sections~\ref{ssec:dyna} and~\ref{sec:dyna}. However, outside of these sections, we will only use results that are known both for face percolation on the hexagonal lattice and bond percolation on the square lattice, for which conformal invariance has not been proven. In Section~\ref{sec:Z2}, we state explicitly the results that we prove for bond percolation on $\Z^2$.

Let us define face percolation on the hexagonal lattice. Let $\mathcal{H}$ be the set of faces of this lattice. An element $\omega \in \{0,1\}^\mathcal{H}$ induces a coloring of the plane as follows: the hexagon $H \in \mathcal{H}$ is colored in black (resp.\ white) if $\omega_H=1$ (resp.\ $\omega_H=0$). Given a parameter $p \in [0,1]$, we equip the space $\{0,1\}^\mathcal{H}$ with the product $\sigma$-algebra and the product probability measure $(p\delta_1+(1-p)\delta_{0})^{\otimes \mathcal{H}}$. Kesten has proven in \cite{Kes80} that the critical parameter of this model is $p_c=1/2$. More precisely, if $p \le 1/2$ then a.s.\ there is no unbounded black component while this is a.s.\ the case if $p > 1/2$. We refer to \cite{Gri99,BR06,Wer07,GS14} for more about planar Bernoulli percolation.

For every integer $n \ge 1$, let $\mathcal{H}_n$ be the set of hexagons that intersect the square $[0,n]^2$ and let $g_n \, : \, \{0,1\}^{\mathcal{H}_n} \rightarrow \{0,1\}$ denote the indicator function of the left-right crossing of $[0,n]^2$, that is, the event that there is a continuous black path included in $[0,n]^2$ that connects the left side of this square and its right side. Since $p_c=1/2$, the measure induced on the hypercube $\{0,1\}^{\mathcal{H}_n}$ by critical percolation is the uniform probability measure $\Proba$ already considered above.

In \cite{BKS99}, Benjamini, Kalai and Schramm have introduced the notion of \textit{noise sensitivity} and have proven that the crossing events are noise sensitive, that is,
\[
\forall t \in (0,1] \quad \COV_t(g_n) \underset{n \rightarrow +\infty}{\longrightarrow} 0.
\]
 In \cite{SS10}, Schramm and Steif have proven the following quantitative noise sensitivity result: There exists $c>0$ such that $\COV_{t_n}(g_n) \rightarrow 0$ for any sequence $(t_n)_{n \ge 1}$ that satisfies $t_n \ge n^{-c}$. The sharp noise sensitivity theorem has then been proven by Garban, Pete and Schramm in \cite{GPS10}. In order to state this result, let us denote by $\alpha_n$ the probability that there exist four arms of alternating colors from scale $1$ to scale $n$ (see the beginning of Section~\ref{ssec:arm} for a more precise definition) and let
\begin{equation*}
\varepsilon_n = \frac{1}{n^2\alpha_n}.
\end{equation*}
(See the next subsection and the sketch of proof -- in Section \ref{ssec:sketch} -- for some intuition behind the importance of the four-arm event in the study of noise sensitivity.)

\begin{theorem}[Sharp noise sensitivity, \cite{GPS10}]\label{thm:GPS}
Let $(t_n)_{n \ge 1} \in [0,1]^{\N^*}$ and let $g_n$ be the indicator of the left-right crossing of $[0,n]^2$. We have
\[
\lim_{n \rightarrow +\infty} \frac{t_n}{\varepsilon_n} = +\infty \quad \implies \quad \lim_{n \rightarrow +\infty}
  \COV_{t_n}(g_n) = 0,
\]
and
\[
\lim_{n \rightarrow +\infty} \frac{t_n}{\varepsilon_n} = 0 \quad \implies \quad
\lim_{n \rightarrow+\infty} \COV_{t_n}(g_n) = \frac 1 4.
\]
\end{theorem}
 
For face percolation on the triangular lattice, it has been proven in \cite{SW01} that $\alpha_n = n^{-5/4+o(1)}$, thus $\varepsilon_n=n^{-3/4+o(1)}$. Except when computing Hausdorff dimensions of exceptional times for dynamical percolation (see Section \ref{ssec:dyna}), we will not need the exact values of these exponents, and we will mostly rely on the following   lower bound on the ratio $\alpha_n/\alpha_m$ (which is known both for face percolation on $\mathcal{H}$ and bond percolation on $\Z^2$, see for instance \cite[Lemma 6.5]{GS14}): There exists $c>0$ such that, for every $n \ge m \ge 1$,
\begin{equation}\label{eq:fourarm_lower_back}
\frac{\alpha_n}{\alpha_m} \ge c \left( \frac{m}{n} \right)^{2-c}.
\end{equation}

The proofs from \cite{BKS99,SS10,GPS10} rely on the study of the \textit{Fourier spectrum of Boolean functions}. Spectral tools are very rich (they are for instance connected to the theories of influences, hypercontractivity, decision trees etc. -- see \cite{GS14,oDo14}), but some limitations appear when dealing with non-monotonic events (see Open Problem 8 of \cite{GPS10})\footnote{However, as explained in \cite[Section 6.2]{GHSS19}, the inclusion-exclusion principle and noise sensitivity properties for multiple monotonic events imply some noise sensitivity results for non-monotonic events.}, and more importantly, the Fourier analysis  have not been developed for other models of statistical mechanics, such as FK-percolation. In the present work, we give a new proof of Theorem~\ref{thm:GPS} without relying on any spectral tool. We also obtain new results for noise sensitivity of non-monotonic events, and we hope that this approach can be used to tackle  noise sensitivity results for Markov dynamics associated to non product measure. The main object of our proof is the \textit{dynamical four-arm event} introduced in Section~\ref{ssec:arm}, that we study via a strategy inspired by Kesten's work on near-critical percolation \cite{Kes87} (see also \cite{Wer07,Nol08}).  

\begin{remark}\label{rem:quantitative_GPS}
In Section~\ref{sec:sharp-noise-sens}, we will actually compute $\COV_t(g_n)$ up to constant factors (see Equation~\eqref{eq:3}). This quantitative  version of  Theorem~\ref{thm:GPS} is equivalent to the main spectral estimate from \cite{GPS10}, that is, Theorem 1.1 therein. Let us note that, on the contrary, the geometric spectral results from \cite{GPS10} (for instance the partial spatial independence and clustering properties, see also \cite{GV19a,GHSS19}) cannot be extracted from such a non-spectral result.
\end{remark}

We refer to \cite{BGS13,For16} for the study of noise sensitivity for another type of noise, namely the exclusion noise. Noise sensitivity has also been studied for other planar percolation models (that are believed to behave similarly to Bernoulli percolation): Poisson Boolean percolation \cite{ABGM14}, Voronoi percolation \cite{AGMT16,AB18} and percolation of nodal lines \cite{GV19b}. This notion has also been studied for Erd{\H{o}}s--Renyi percolation \cite{LS15,RS18,LP20} and bootstrap percolation \cite{BP15}. See the books \cite{GS14,oDo14} for more about noise sensitivity and Boolean functions.

\subsection{Arm events}\label{ssec:arm}

If $n \ge m \ge 1$, we let $\mathcal{H}_{m,n}$ denote the set of hexagons that intersect the annulus $[-n,n]^2 \setminus (-m,m)^2$ and we call an \defn{arm of type $1$ (resp.\ type $0$) from $m$ to $n$}  a continuous black  (resp.\ white) path from $\partial[-m,m]^2$ to $\partial[-n,n]^2$ included in the annulus $[-n,n]^2 \setminus (-m,m)^2$. For  $j\ge  1$, $\sigma  \in \{0,1\}^j$ and $n \ge m \ge j$, define 
\[
f_{m,n}^\sigma:\{0,1\}^{\mathcal{H}_{m,n}}\to \{0,1\}
\]
to be the  indicator of the event that there exist $j$ disjoint\footnote{Here, disjoint means that no hexagon is interseted by two different arms.} arms from $m$ to $n$ with respective types $\sigma_1,\ldots,\sigma_j$, in the counter-clockwise order. Such an event is called an \textit{arm event}. See Figure~\ref{fig:4arm} for an illustration of the alternating four-arm event.

Note that the assumption $n \ge m \ge j$ implies that $\Ex [ f_{m,n}^\sigma ] > 0$. If $1 \le m \le n < j$ or $1 \le n < m$, we use the convention $f_{m,n}^\sigma \equiv 1$; if $1 \le m < j \le n$, we let $f_{m,n}^\sigma = f_{j,n}^\sigma$. Define also
\[
f_{n}^\sigma=f_{j,n}^\sigma.
\]

\begin{figure}[h!]
\begin{center}
\includegraphics[scale=0.4]{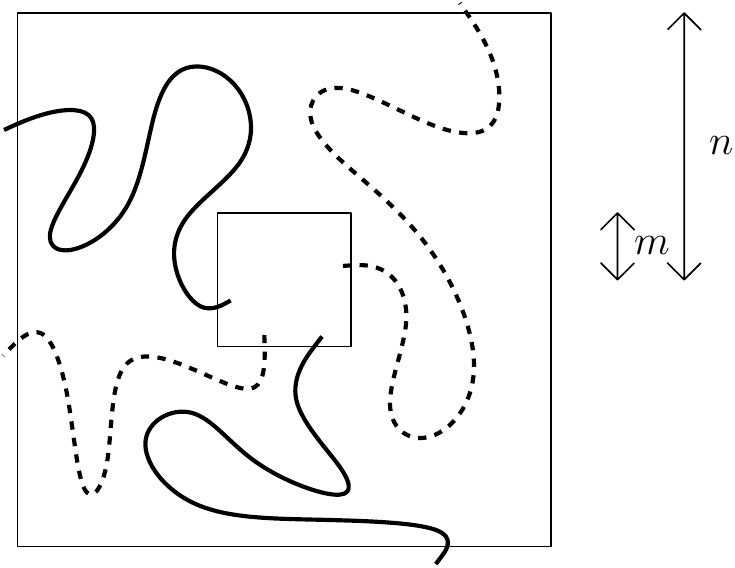}
\caption{\small The alternating four-arm event from $m$ to $n$, which corresponds to $\sigma = 0101$, or equivalently to $\sigma=1010$.\label{fig:4arm}}
\end{center}
\end{figure}

Similarly, define $f_{m,n}^{\sigma+}$ (resp.\ $f_{n}^{\sigma+}$) to be the indicator of the same event with the additional constraint that the arms stay in the (upper) half-plane, and the arm types are from right to left. Given an \emph{arm type} $\star$ (i.e.\ $\star = \sigma$ or $\sigma+$), we define 
\[
\alpha^\star_{m,n}=\Ex [ f^\star_{m,n} ],\quad \alpha^\star_{n} =\Ex [ f^\star_{n} ],
\]
and for every $t\in[0,1]$
\[
\pi^\star_{m,n}(t)= \Qu_t(f^\star_{m,n}),\quad \pi^\star_{n}(t)= \Qu_t(f^\star_{n}).
\]
Since we will often see $t$ as a time parameter, we call $\pi^\star_{m,n}(t)$ a \textit{dynamical arm event} probability. Note that we have the following ``boundary conditions'':
\[
\pi^{\star}_{m,n}(0)=\alpha^{\star}_{m,n}, \quad \pi^{\star}_{m,n}(1) = (\alpha^{\star}_{m,n})^2.
\]
The study of the (alternating) four-arm event, which corresponds to the arm type $\star=0101$, is at the core of our proof. This event corresponds to the geometric interpretation of pivotality for crossing events (see the beginning of Section \ref{sec:russo} for the definition of pivotality). When considering this event, we will omit the exponent in our notation and simply write 
\[
\alpha_{m,n} = \alpha^{0101}_{m,n}, \quad \pi_{m,n}=\pi^{0101}_{m,n}.
\]
(And similarly for $\alpha_n$ and $\pi_n$.) One can note that $\pi_n(t)$ is the probability that the origin has a four-arm both at times $0$ and $t$.
\medskip

For  each $t \in (0,1]$, we let
\begin{equation}\label{eq:sens_length}
\ell(t)=\min\{ n \ge 1 \, : \, n^2\alpha_n \geq 1/t \}.
\end{equation}
By \eqref{eq:fourarm_lower_back} with $m=1$, $n^2\alpha_n \rightarrow +\infty$. As a result, $\ell(t)$ is a well defined positive integer that goes to $+\infty$ as $t$ goes to $0$. We let $\ell(0)=+\infty$. By analogy to the correlation length studied in  \cite{Kes87}, we call $\ell(t)$ the \emph{sensitivity length}, and it has the following interpretation (see Theorem \ref{thm:GPS}): If $n \ll \ell(t)$, then crossing events at scale $n$ are not affected by a noise of intensity t. On the contrary, if $n \gg \ell(t)$, a crossing event at scale $n$ and its noised version behave roughly independently.

\begin{remark}
At this point, the link between our definition of the sensitivity length (which is made by using only the four-arm probability) and the interpretation mentioned above may not be clear. One of our main goals is actually to make this connection. More precisely: If $n \ll \ell(t)$, then a first moment argument shows that no pivotal is affected by the noise and therefore one can indeed expect that connection properties do not change at scale $n$. On the contrary, if $n \gg \ell(t)$, pivotal points start to be affected by the noise and therefore crossing events start to be affected themselves. The key difficulty is to show that they actually behave (almost) independently.
\end{remark}

The following theorem in the cases $\star=1$ and $\star=1+$ is a direct consequence of \cite{GPS10} (more generally, the proof from \cite{GPS10} very likely extends to any monochromatic arm event) but is new for the other arm types. The reason is that spectral techniques have some limitations in the study of non-monotonic events (see Open Problem 8 of \cite{GPS10}).
\begin{theorem}\label{thm:arm}
Let $\star$ be any arm type and let $\ell$ be the sensitivity length defined in \eqref{eq:sens_length} above. There exist $c,C>0$ such that for every  $t \in [0,1]$, we have
\begin{itemize}
\item[1)] (stability region) for every  $\ell(t) \ge n \ge m \ge 1$,
  \begin{equation*}
c\alpha^\star_{m,n}\le \pi^{\star}_{m,n}(t) \le \alpha^{\star}_{m,n};
\end{equation*}
\item[2)] (sensitivity region) for every $n\ge m  \ge \ell(t)$,
    \begin{equation*}
     (\alpha^{\star}_{m,n})^2\le \pi^{\star}_{m,n}(t)   \le C(\alpha^{\star}_{m,n})^2.
    \end{equation*}
\end{itemize}
\end{theorem}
Theorem~\ref{thm:arm} is proven in Section~\ref{sec:sharp}. In Section~\ref{sec:sharp-noise-sens}, we deduce Theorem~\ref{thm:GPS} from Theorem~\ref{thm:arm}.

\subsection{Exceptional times for dynamical percolation}\label{ssec:dyna}

In this section, we rely on the computation of the critical exponents for face percolation on $\mathcal{H}$ from \cite{LSW02,SW01}, see Section~\ref{sec:dyna}. We refer to Section~\ref{sec:Z2} for the study of exceptional times for bond percolation on $\Z^2$, for which these exponents are not known. Let $\omega(0)$ be a critical percolation configuration on the hexagonal lattice, i.e.\
\[
\omega(0) \sim \left( \frac{\delta_1+\delta_{0}}{2} \right)^{\otimes \mathcal{H}}
\]
(recall that $\mathcal{H}$ is the set of faces of the hexagonal lattice). We study the Markov process $t \in \R_+ \mapsto \omega(t)$ obtained by resampling the colours at rate $1$. This process was introduced in 1992 by Benjamini (unpublished) and independently in \cite{HPS97}. It follows from known results on (static) percolation that for any  fixed $t$, a.s. there is  no unbounded monochromatic component in $\omega(t)$. However, Schramm and Steif \cite{SS10} proved that  there exist random times, called \textit{exceptional times},  at which such a property holds. Later, Garban, Pete and Schramm \cite{GPS10} computed the exact Hausdorff dimension of the set of exceptional times, and also proved the existence of exceptional times for  percolation on $\Z^2$.

We  have the following theorem about the Hausdorff dimensions of exceptional times. The upper bounds on these four dimensions have been computed in \cite{SS10} and the lower bound for Items 1 and 2 have been computed in \cite{GPS10}. In Section~\ref{sec:dyna}, we deduce the four lower bounds from Theorem~\ref{thm:arm}. Let us note that it has been proven in \cite{GPS10} that the Hausdorff dimension from Item 3 is at least $1/9$ and that it was not known previously that the set of times from Item 4 was non-empty.

\begin{theorem}\label{thm:Hausdorff}
The following holds a.s.:
\begin{itemize}
\item[(1)] the Hausforff dimension of the set of times with an unbounded black component is $31/36$;
\item[(2)] the Hausdorff dimension of the set of times with an unbounded black component in the upper half-plane is $5/9$;
\item[(3)] the Hausforff dimension of the set of times with both a black and a white unbounded components is $2/3$;
\item[(4)] the Hausdorff dimension of the set of times with two disjoint infinite black paths and an unbounded white component is $1/9$.
\end{itemize}
\end{theorem}
\begin{remark}\label{rem:exp_dyna_jge4}
As proven in \cite{SS10}, a.s.\ there is no time with two disjoint unbounded black components (which is equivalent to the fact that a.s.\ there is no time with four infinite paths of alternating colors) and no time with both a black and a white unbounded component in the half-plane. The following quantitative version of this result is a consequence of \cite{SS10} and Theorem~\ref{thm:arm}. (More precisely, the upper bounds are consequences of \cite[Section 8]{SS10} and one can prove the lower bounds by using Theorem~\ref{thm:arm} and a second moment method analogous to that of the proof of Theorem~\ref{thm:Hausdorff}. We leave the details to the reader.) If $j \geq 4$ and $\sigma \in \{0,1\}^j$ is a sequence which is not monochromatic, then
\begin{align*}
&\Pro \left[ \exists t \in [0,1], \, f^\sigma_n(\omega(t))=1 \right] = n^{-(j^2-10)/12+o(1)}.
\end{align*}
If $j \geq 2$ and $\sigma \in \{0,1\}^j$ is a sequence which is not monochromatic, then
\begin{align*}
&\Pro \left[ \exists t \in [0,1], \, f^{\sigma+}_n(\omega(t))=1 \right] = n^{-(2j^2+2j-9)/12+o(1)}.
\end{align*}
\end{remark}
\medskip

In \cite{HPS15}, Hammond, Pete and Schramm have proven that an unbounded component at a ``uniformly chosen'' exceptional time has the same law as the so-called Incipient Infinite Cluster. The existence of exceptional times has also been proven for percolation under (long range) exclusion dynamics, see \cite{GV19a}, and for Voronoi percolation in a dynamical environment, see \cite{Van19}. The scaling limit of dynamical percolation has been defined and studied in \cite{GPS13,GPS18}. Dynamical percolation has also been studied in a Liouville environment, see \cite{GHSS19}. See \cite{Ros17} for the study of the scaling limit of dynamical Erd{\H{o}}s--Renyi percolation.

\subsection{Organization of the paper}

The paper is organized as follows. In Section~\ref{sec:sketch}, we present a sketch of proof of the sharp noise sensitivity theorem and explain some connections with other works. In Section~\ref{sec:russo}, we prove a dynamical Margulis--Russo formula. In Section~\ref{sec:qm} we extract a quasi-multiplicativity property for $\pi_{m,n}(t)$ from \cite[Section 5]{GPS10}. In Section~\ref{sec:diff}, we prove the differential inequalities that we will use in the next three sections. Our main intermediate result is that the dynamical four-arm event has superquadratic decay above the sensitivity length. This result is stated and proven in Section~\ref{sec:new-proof-sharp}. The sharp noise sensitivity results Theorems~\ref{thm:GPS} and~\ref{thm:arm} are proven in Sections~\ref{sec:sharp} and~\ref{sec:sharp-noise-sens} respectively, and Theorem~\ref{thm:Hausdorff} is proven in Section~\ref{sec:dyna}. Finally, in Section~\ref{sec:Z2}, we discuss consequences for bond percolation on $\Z^2$.

\paragraph{Convention for the constants.}
In all the paper, $c$ and $C$ are positive constants \textit{that are not necessarily the same at each occurrence}. These constants are independent of the scale parameters $k,m,n,\ldots$, the time parameters $s,t,u,\ldots$, the coordinates $i$ and the subsets of coordinates $R,S,\ldots$, but may depend on the arm type $\star$. Following this convention, we will often omit the quantifiers for the constants. For example, the statement ``$\forall t$ $\psi(t)\ge c\varphi(t)$'' should be interpreted as  ``$\exists c>0\ \forall t$ $\psi(t)\ge c\varphi(t)$''. Occasionally, we may we need to fix a constant for multiple uses, and we will make it clear by  writing the constant with an index (e.g. $C_1,C_2,\ldots$). Finally, given two functions $\varphi,\psi$ with non-negative values, we let $$\varphi \asymp \psi$$ if $c\varphi \le \psi \le C\varphi$.

\paragraph{Acknowledgments.} We are extremely grateful to Gábor Pete for very inspiring discussions about instability properties for the pivotal set. Moreover, we would like to thank Laurin K\"{o}hler-Schindler, Stephen Muirhead and Alejandro Rivera for inspiring discussions, and Christophe Garban for helpful comments on a previous version of this paper. Finally, we wish to thank two anonymous referees for helpful comments.

This project has received funding from the European Research Council (ERC) under the European Union’s Horizon 2020 research and innovation program (grant agreement No 851565). Both authors are supported by the NCCR SwissMAP funded by the Swiss NSF; HV is supported by the SNF Grant
No 175505.

\section{Discussion and sketch of proof}\label{sec:sketch}

\subsection{Sketch of proof}\label{ssec:sketch}

In this section, we present a sketch of proof of the sharp noise sensitivity results Theorems~\ref{thm:GPS} and~\ref{thm:arm}. Recall that $g_n$ is the indicator of the left-right crossing of $[0,n]^2$. In order to study the effect of the noise on the Boolean function $g_n$, we will use a dynamical version of Margulis--Russo's formula (see Section~\ref{sec:russo}) to show that $-\frac{d}{dt}\Qu_t(g_n) \asymp  n^2\pi_n(t)$. This formula can be understood as the analogue for noise sensitivity of the formula used to express the derivative in the parameter $p$ of crossing probabilities in terms of the four-arm probability. This implies the following identity about the covariance of crossing events:
\begin{equation}\label{eq:draft_cov}
\COV_t(g_n) = \Qu_t(g_n)-\Qu_1(g_n) \asymp \int_t^1 n^2\pi_n(s)ds.
\end{equation}
Equation \eqref{eq:draft_cov} enables one to reduce the analysis of $\COV_t(g_n)$ to that of $\pi_n(t)$, and in the rest of this section we explain how to derive a precise expression of $\pi_n(t)$ up to constants. To this purpose, we rely on an approach inspired by Kesten's study of the near-critical four-arm event \cite{Kes87} (see also \cite{Wer07,Nol08}). A key idea is to apply again Margulis--Russo's formula to the four-arm probability, which gives rise to the following inequality for all $n \ge m \ge 1$ and $t \in [0,1]$:
\begin{equation}\label{eq:draft_fourarm}
0\le -\frac{d}{dt} \log\left( \pi_{m,n}(t) \right) \le C\sum_{k=m}^n k\pi_k(t). 
\end{equation}
\begin{remark}\label{rem:mono_pi}
By \eqref{eq:draft_fourarm}, the quantities $\pi_{m,n}(t)$ are monotonic in $t$ (even if the four-arm event is not monotomic), which is a helpful property that has no analogue in Kesten's approach.
\end{remark}
By quasi-multiplicativity (proven in \cite{GPS10} for dynamical arm events, see Section~\ref{sec:qm} for more details), we have $\pi_{m,n}(t) \asymp \frac{\pi_n(t)}{\pi_m(t)}$, and therefore \eqref{eq:draft_fourarm} can be thought of as a PDI (Partial Differential Inequality) of the two-variable function $(t,n) \mapsto \pi_n(t)$ defined in the domain $D=[0,1] \times \N^*$. This function is known at the boundary of the domain:
\[
\pi_n(0) = \alpha_n, \quad \pi_n(1)=\alpha_n^2, \quad \pi_1(t)=1,
\]
and we will use the PDI (together with other properties of this function) to compute it in the interior $D$.

Recall that $\varepsilon_n = \frac{1}{n^2\alpha_n}$ and that $\ell(t)$ is the sensitivity length defined in \eqref{eq:sens_length}. In the study of the PDI \eqref{eq:draft_fourarm}, two domains appear naturally: the domains below and above the curve $t \mapsto \ell(t)=t^{-4/3}$ (see Figure~\ref{fig:curve}), where for simplicity we have assumed that $\alpha_n=n^{-5/4}$, which implies that $\varepsilon_n=n^{-3/4}$ and $\ell(t)=t^{-4/3}$. (This is of course not exactly true, but justified by the fact that Smirnov and Werner \cite{SW01} have proven that $\alpha_n = n^{-5/4+o(1)}$ for critical percolation on the hexagonal lattice.)

\begin{figure}[h!]
\begin{center}
\includegraphics[scale=0.5]{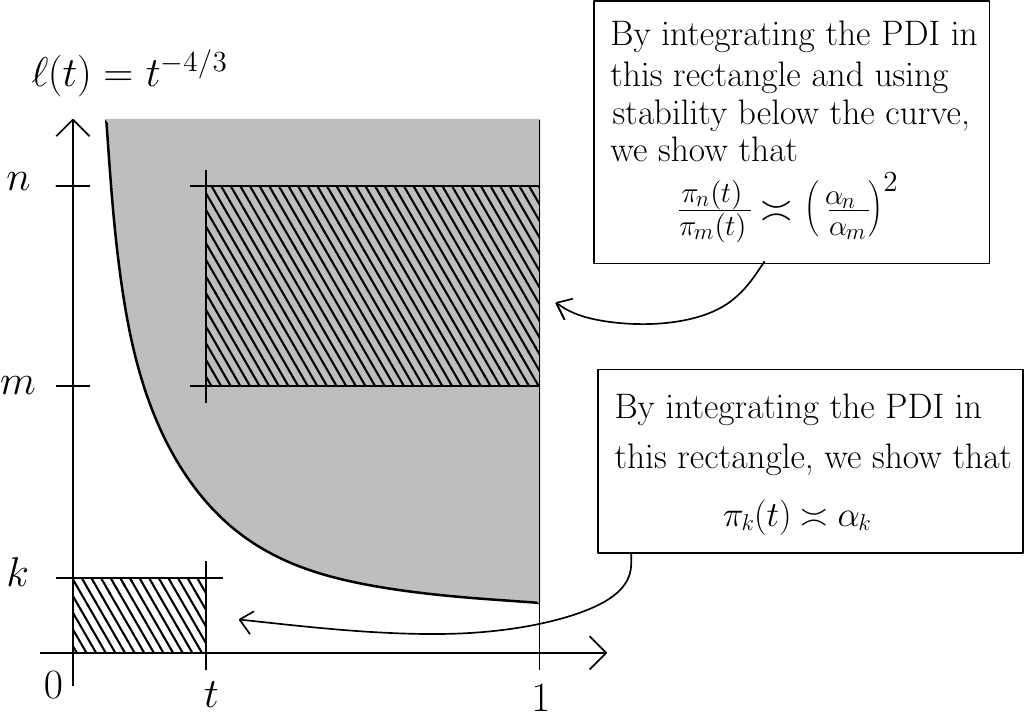}
\caption{\small The curve $t \mapsto \ell(t)$. We show stability of the dynamical four-arm probability below the curve (i.e.\ in the white region) and stability of the ratio of two dynamical four-arm probabilities above the curve (i.e.\ in the grey region).\label{fig:curve}}
\end{center}
\end{figure}

Using the PDI \eqref{eq:draft_fourarm} with $m=1$, it follows from the definition of $\ell(t)$ that $\pi_n(t)$ is stable below the curve, that is
\begin{equation}\label{eq:draft_stab}
\forall n \le \ell(t) \quad \pi_n(t) \asymp \alpha_n=n^{-5/4}.
\end{equation}
This step is exactly the same as in Kesten's approach. The main difference is the study above the curve. Indeed, if $p\neq p_c$, renormalization arguments imply that the percolation probabilities decay exponentially fast above the correlation length, while in the present context, we always stay at criticality and $\pi_n(t)$ will also decay polynomially fast above the sensitivity length.

For simplicity, let us assume that there exists some exponent $\beta>0$ such that, for all $n \ge m \ge \ell(t)$,
\[
\frac{\pi_n(t)}{\pi_m(t)} \asymp \left( \frac{m}{n} \right)^\beta.
\]
By \eqref{eq:draft_stab}, we already know that $\pi_n(t) \asymp \alpha_n$ on the curve (i.e.\ $\pi_{\ell(t)}(t) \asymp \alpha_{\ell(t)}$). Together with  quasi-multiplicativity, this implies that for every $(t,n)$ above the curve,
\[
\pi_n(t) \asymp \alpha_{\ell(t)} \left( \frac{\ell(t)}{n} \right)^\beta = t^{\frac 5 3 -\frac 4 3 \beta} n^{-\beta}.
\]
The key step is to prove that the exponent $\beta$ above the curve is strictly larger than $2$ (contrary to the exponent $5/4$ below the curve, see \eqref{eq:draft_stab}). Under the assumption that this exponent $\beta$ exists, this follows directly from
\begin{equation}\label{eq:COV_cha_draft}
1 \ge \COV_{\varepsilon_n}(g_n) \overset{\eqref{eq:draft_cov}}{\ge} c \int_{\varepsilon_n}^1 n^2 \pi_n(t)dt.
\end{equation}
Of course, the existence of $\beta$ is a very strong assumption. However, in Section~\ref{sec:new-proof-sharp} (which is the key section of the paper), we will replace this assumption by the known quasi-multiplicativity and monotonicity properties of $\pi_{m,n}(t)$, which together with \eqref{eq:COV_cha_draft} will enable us to show that for all $n \ge m \ge \ell(t)$,
\begin{equation}\label{eq:2+c_draft}
\frac{\pi_n(t)}{\pi_m(t)} \le C \left( \frac{m}{n} \right)^{2+c}.
\end{equation}

By plugging \eqref{eq:2+c_draft} in the PDI \eqref{eq:draft_fourarm}, one can then show that the ratio $\pi_n(t)/\pi_m(t)$ is stable above the curve. Therefore,
\[
\forall n \ge m \ge \ell(t) \quad \frac{\pi_n(t)}{\pi_m(t)} \asymp \frac{\pi_n(1)}{\pi_m(1)} = \frac{\alpha_n^2}{\alpha_m^2},
\]
which together with \eqref{eq:draft_stab} concludes the desired computation of $\pi_n(t)$.

\begin{remark}
Our main intermediate result is \eqref{eq:2+c_draft}: $\pi_n(t)$ is \textit{superquadratic} above the curve. Note that in near-critical percolation \cite{Kes87}, even the (black or white, depending on if $p<p_c$ or $p>p_c$) one-arm event probability is superquadratic above the curve since it is exponentially small. In order to prove \eqref{eq:2+c_draft}, we rely on several properties of $\pi_n(t)$, including properties that are not present in Kesten's proof, such as the monotonicity property (see Remark~\ref{rem:mono_pi}) and quasi-multiplicativity above the curve.
\end{remark}

\subsection{Further connections with other works}

\paragraph{A dynamical Margulis--Russo formula.} The dynamical Margulis--Russo formula used in the present paper (see Section~\ref{sec:russo}) is not new. In particular, this is a direct consequence of general covariance representations along Markov semi-groups, see for instance \cite[Section 1.7.1 and Chapter 2]{BGL13}. Moreover, as explained in \cite{GP20} where Galicza and Pete study sparse reconstruction of Boolean functions, this is an analogue for Boolean functions of covariance formulas for the Ornstein--Ulhenbeck dynamics, which are for instance central in Chatterjee's theory of superconcentration and chaos \cite{Cha14} (see e.g.\ Section 4 of Chapter 2 from this book). Analogous covariance formulas also play an important role in Pitterbarg's work \cite{Pit96}, which has been reinterpreted by using pivotal events (in a Gaussian setting) in \cite{RV19,BMR18}.

As explained in \cite[Section 5.4]{BKS99}, $\COV_t(g_n)$ can be seen as the variance of the quenched probability of the crossing event on a random graph. From this point view, the dynamical Margulis--Russo formula is very similar to Theorem 2.1 of \cite{AGMT16} that is used to study the quenched crossing probabilities in Voronoi percolation.

\paragraph{Noise sensitivity via differential equations.} In \cite{EG20}, Eldan and Gross study noise sensitivity (and other properties of Boolean functions) by using differential inequalities satisfied by $\COV_t(f)$. Contrary to the present paper, their techniques are based on stochastic calculus and they study Boolean functions in general.

\paragraph{Upper and lower bounds on the pivotal probabilities.} We refer to the beginning of Section~\ref{sec:russo} for the notion of pivotality and we recall (see for instance \cite{Wer07}) that, if we neglect boundary issues, $\forall i \quad \Proba[i \text{ is piv. for } g_n] \asymp \alpha_n$. The classical result  \cite[Theorem~1.3]{BKS99} states that the sequence $g_n$ is noise sensitive if and only if $\sum_{i}\Proba[i \text{ is piv. for } g_n]^2\asymp n^2\alpha_n^2$ converges to $0$ as $n\to\infty$. And therefore, noise sensitivity follows from an upper bound on $\alpha_n$. On the contrary, one of the main inputs of our approach is the lower bound \eqref{eq:fourarm_lower_back} on $\alpha_n$, which essentially means that there are many pivotals in expectation. 

\paragraph{Differential inequalities in percolation.} As mentioned earlier, differential inequalities were used by Kesten to study the phase transition of percolation, but the interest of differential inequalities for percolation was also highlighted in several other works, see for instance \cite{AB87} or the study of mean-field behavior \cite{HH17}.

\section{A dynamical Margulis--Russo formula}\label{sec:russo}

In this section we prove a dynamical version of the Margulis--Russo formula, in the general  framework of Boolean functions on the hypercube. Given a general Boolean function $f:\{0,1\}^n\to\{0,1\}$, we will express the  derivative of $\Qu_t(f)$ with respect to the noise parameter $t$ in terms of the pivotal events. This formula will be instrumental in Section~\ref{sec:diff}, where we will establish differential inequalities satisfied by the crossing and arm probabilities.

We consider a heterogeneously noised configuration where the points are resampled with probability $t$ inside a subset $S\subseteq [n]:=\{1,\cdots,n\}$, and fully resampled outside $S$.\footnote{This will be useful later when studying the derivative of arm probabilities. The set $S$ will be chosen as the set of points away from the boundary of the considered domain, and it will allow us to ignore the effect of points close the boundary for which the pivotality has a different interpretation than the point in the ``bulk''.} Let $n \ge 1$, $t \in [0,1]$ and $S \subseteq [n]$. Also, let $X=(X_1,\cdots,X_n) \sim \Proba$ and let $Y=(Y_1,\cdots,Y_n) \sim \Proba$ be obtained from $X$ by resampling independently each coordinate $i \in S$ with probability $t$ and by resampling independently each coordinate $i \in [n] \setminus S$ with probability $1$. If $f \, : \, \{0,1\}^n \rightarrow \{0,1\}$, we let
\[
\Qu_t^S(f) = \E [ f(X)f(Y) ], \quad \COV_t^S(f)=\Cov(f(X),f(Y)).
\]
In particular, $\Qu_t^{[n]}(f)=\Qu_t(f)$ and $\COV_t^{[n]}(f)=\COV_t(f)$. If $A \subseteq \{0,1\}^n$, we let 
\[
\Qu_t^S(A)=\Qu^S_t(\textbf{1}_A).
\]

It is known that if $Y$ is a noised  configuration of $X$, then the covariance between $f(X)$ and $f(Y)$ decreases when the noise increases (see \cite{GS14,oDo14}). In the present work, we will use the following monotonicity properties of $\Qu_t^S(f)$ (that is deduced below from general derivative formula with respect to the noise).
\begin{proposition}\label{prop:2}
  Let $f \, : \, \{0,1\}^n \rightarrow \{0,1\}$. The quantity $\Qu_t^S(f)$ is non-increasing in $t$ and non-decreasing in $S$.
\end{proposition}
In order to state the main result of this section, let us recall  the notion of pivotality associated to a Boolean function. Let $n \ge 1$ and $f \, : \, \{0,1\}^n \rightarrow \{0,1\}$.  A coordinate $i \in [n]$ is said to be  \defn{pivotal} for $f$ in a configuration $x \in \{0,1\}^n$ if $f(x^i) \neq f(x)$, where $x^i \in \{0,1\}^n$ is obtained from $x$ by changing the $i^{th}$ coordinate.

\begin{proposition}[Dynamical Margulis--Russo formula]
  \label{prop:1}
  Let $f:\{0,1\}^n\to\{0,1\}$ and $S\subseteq [n]$. For every  $t\in[0,1]$, we have
\begin{equation*}
0\le -\frac{d}{dt}\Qu_t^S(f)\le \frac14\sum_{i\in S} \Qu_t^S(i \text{ is piv. for } f),
\end{equation*}
and the second inequality is an equality when $f$ is monotone.
\end{proposition}

\begin{remark}
More general versions of these propositions are known (for instance when the measure is still product but not uniform), see e.g. \cite{BGL13}.
\end{remark}

In order to prove Proposition~\ref{prop:1}, we use the following coupling. Let $S \subset [n]$. Let $X \sim \Proba$, let $Z \sim \Proba$ independent of $X$ and let $U=(U_1,\cdots,U_n)$ be a uniform variable in $[0,1]^n$ independent of $(X,Z)$. Given some $(t_1,\cdots,t_n) \in [0,1]^n$, let $Y(t_1,\cdots,t_n) \sim \Proba$ be defined by
\[
Y_i(t_1,\cdots,t_n)=
\begin{cases}
\text{$X_i$ if $U_i \ge t_i$ and $i \in S$};\\
\text{$Z_i$ if $U_i < t_i$ or $i \notin S$.}
\end{cases}
\]

For every $f \, : \, \{0,1\}^n \rightarrow \{0,1\}$ and  $i \in [n]$, the discrete partial derivative of $f$ along the $i^{\text{th}}$ coordinate is defined as follows:
\[
\nabla_i f = \frac{1}{2}\left(f \circ \sigma_i^1 - f \circ \sigma_i^0 \right),
\]
where $\sigma_i^1 \, : \, \{0,1\}^n \rightarrow \{0,1\}^n$ (resp.\ $\sigma_i^0 \, : \, \{0,1\}^n \rightarrow \{0,1\}^n$) replaces the $i^\text{th}$ coordinate by $1$ (resp.\ $0$) and leaves the other coordinates unchanged.
\begin{lemma}\label{lem:diff}
Let $f \, : \, \{0,1\}^n \rightarrow \{0,1\}$. For every $i\in[n]$, we have
\begin{equation}
-\frac{\partial}{\partial t_i}\E [ f(X)f(Y(t_1,\ldots,t_n))] =
\begin{cases}
\E \left[ \nabla_i f(X) \nabla_i f(Y(t_1,\ldots,t_n)) \right] \text{ if $i \in S$};\\
0 \text{ if $i \notin S$}
\end{cases}
\label{eq:6}
\end{equation}
and this quantity is non-negative. 
\end{lemma}

Before proving this lemma, let us see how it concludes the proof of Propositions~\ref{prop:2} and~\ref{prop:1}. First, the non-negativity of \eqref{eq:6} implies that the quantity
\begin{equation*}
  \E [ f(X)f(Y(t_1,\ldots,t_n))]
\end{equation*}
is non-decreasing in each coordinate $t_i$, from which Proposition~\ref{prop:2} follows. For the  proof of Proposition~\ref{prop:1}, we use the chain rule and Lemma~\ref{lem:diff} to obtain that
\begin{equation*}
-\frac{d}{dt} \Qu_t^S(f)=\sum_{i\in S}\Qu_t^S(\nabla_if).
\end{equation*}
Then, the proposition follows from the fact that  $2|\nabla_i f| = \textbf{1}_{i \text{ is piv. for } f}$ and, if we assume furthermore that $f$ is increasing, $2\nabla_i f = \textbf{1}_{i \text{ is piv. for } f}$.

\begin{proof}[Proof of Lemma~\ref{lem:diff}]
The result for $i \notin S$ comes from the fact that $Y(t_1,\dots,t_n)$ does not depend on $t_i$ in this case. Let us assume that $1 \in S$ and let us prove the lemma for $i=1$. We first observe that, for each $f \, : \, \{0,1\}^n \rightarrow \{0,1\}$, $f(x)=(2x_1-1) \nabla_1f(x)+\Ex_1f(x)$ where
\[
\Ex_1f(x)=\frac{f \circ \sigma_1^1 + f \circ \sigma_1^0}{2}(x).
\]
Let us fix a function $f \, : \, \{0,1\}^n \rightarrow \{0,1\}$ and some $t_1,\ldots,t_n\in[0,1]$. Our goal is to compute the derivative at $t_1$ of the function defined by  
\[
\forall s\in[0,1] \quad \varphi(s)= \E \left[ f(X)f(Y(s,t_2,\ldots,t_n)) \right].
\]
Let $\delta \in [-1,1]$ such that $t_1 + \delta \in [0,1]$. From now, we write $Y=Y(t_1,t_2,\ldots,t_n)$ and $Y'=Y(t_1+\delta,t_2,\cdots,t_n)$. Then,
\[
\varphi(t_1+\delta)-\varphi(t_1)=\E \left[ f(X) \left(  f(Y')- f(Y) \right) \right].
\]
Since $f(x)=(2x_1-1) \nabla_1 f(x)+\Ex_1 f(x)$ and since $\nabla_1 f(x)$ and $\Ex_1 f(x)$ do not depend on $x_1$, we have
\[
f(Y')- f(Y)= 2\left( Y_1' - Y_1 \right) \nabla_1 f(Y).
\]
By using that $\nabla_1 f(x)$ and $\Ex_1f(x)$ do not depend on $x_1$ and that $\E [Y_1' - Y_1] = 0$, we then obtain that
\begin{align*}
\varphi(t_1+\delta)-\varphi(t_1) & = \E \left[ \big( (2X_1-1) \nabla_1 f(X) + \Ex_1f(X) \big) \cdot 2 \left( Y_1' - Y_1 \right) \nabla_1f(Y) \right]\\
& = 4 \E \left[ X_1 \left( Y_1' - Y_1 \right) \right] \E \left[ \nabla_1f(X) \nabla_1f(Y) \right]\\
& = -\delta  \E \left[ \nabla_1f(X) \nabla_1f(Y) \right].
\end{align*} 
As a result,
\[
-\frac{\partial}{\partial t_1}\E [ f(X)f(Y(t_1,\ldots,t_n))]=-{\varphi'} (t_1) = \E \left[ \nabla_1 f(X) \nabla_1 f(Y) \right].
\]
This concludes the first part of the lemma.

To prove that the right-hand side of \eqref{eq:6} is non-negative, we introduce another coupling of the law of $(X,Y)$ (where $Y=Y(t_1,\cdots,t_n)$ as above): We define $s_i \in [0,1]$ by $1-t_i=(1-s_i)^2$ and we let $\overline{X} \sim \Proba$. Moreover, we let $\overline{W}$ be obtained from $\overline{X}$ by resampling independently each coordinate $i \in S$ with probability $s_i$ and each coordinate $i \notin S$ with probability $1$. Furthermore, we let $\overline{Y}$ be obtained from $\overline{W}$ by resampling independently each coordinate $i \in S$ with probability $s_i$ and each coordinate $i \notin S$ with probability $1$. Then, $(\overline{X},\overline{Y})$ is a coupling of the law of $(X,Y)$. Moreover, conditionally on $\overline{W}$, $\overline{X}$ and $\overline{Y}$ are independent and have the same distribution. Therefore,
\begin{equation*}
\E \left[\nabla_i f(X) \nabla_i f(Y)\right]=\E \left[\nabla_i f(\overline{X}) \nabla_i f(\overline{Y}) \right]=\E\left[ \E \left[\nabla_i f(\overline{X}) \mid \overline{W} \right]^2\right]
\end{equation*}
is non-negative. (In other words, this non-negativity property comes from the reversibility of the dynamics.)
\end{proof}

\section{Quasi-multiplicativity for dynamical arm events}\label{sec:qm}

In this section, we state a quasi-multiplicativity property for $\pi_{m,n}^\star(t)$ that will be crucial for us and is a direct consequence of quasi-multiplicativity results from \cite{GPS10} (see Lemma~\ref{lem:qm} below). We refer to \cite{Kes87,Nol08,Wer07,SS10} for proofs of quasi-multiplicativity for $\alpha_{m,n}^\star$ (i.e.\ Proposition~\ref{prop:qm} in the special case $t=0$).
\begin{proposition}[Dynamical quasi-multiplicativity]\label{prop:qm}
Let $\star$ be an arm type. There exist $c,C>0$ such that, for every $t \in [0,1]$ and every $n \ge m \ge k \ge 1$,
\begin{align*}
& c \pi^\star_{k,n}(t) \le \pi^\star_{k,m}(t) \pi^\star_{m,n}(t) \le C\pi^\star_{k,n}(t).
\end{align*}
In particular, there exist $c,C>0$ such that, for every $t \in [0,1]$ and every $n \ge m \ge 1$,
\begin{align}
& c\frac{\pi^\star_{n}(t)}{\pi_{m}^\star(t)} \le \pi^\star_{m,n}(t) \le C \frac{\pi^\star_{n}(t)}{\pi_{m}^\star(t)}.\label{eq:38fractional}
\end{align}
\end{proposition}
Recall that $\mathcal{H}_{k,n}$ is the set of hexagons that intersect $[-n,n]^2 \setminus (-k,k)^2$ and that $\mathcal{H}_n$ is the set of hexagons that intersect $[0,n]^2$. Recall also the notation $\Qu_t^S(f^\star_{m,n})$ from the beginning of Section~\ref{sec:russo}. Proposition 5.9 and Equation (5.10) of \cite{GPS10} imply the following lemma.
\begin{lemma}[\cite{GPS10}]\label{lem:qm}
Let $\star$ be an arm type. There exist $c,C>0$ such that, for every $t \in [0,1]$, every $n \ge m \ge k \ge 1$ and every $S \subseteq \mathcal{H}_{k,n}$,
\begin{align}\label{eq:qm_dyna}
& c\Qu_t^S(f^\star_{k,n}) \le \Qu_t^S(f^\star_{k,m}) \Qu_t^S(f^\star_{m,n}) \le C\Qu_t^S(f^\star_{k,n}).
\end{align}
Moreover, there exists $c>0$ such that, for every $t \in [0,1]$, every $n \ge 1$, every $i \in \mathcal{H}_n$ that is at distance at least $n/3$ from the sides of $[0,n]^2$, and every $S \subseteq \mathcal{H}_n$,
\[
\Qu_t^S(i \text{ is piv. for } g_n) \ge c \Qu_t^S(f_n^{0101}).
\]
\end{lemma}
\begin{proof}
The left-hand inequality of \eqref{eq:qm_dyna} follows from independence on disjoint sets (the constant $c$ comes from the fact that in order to obtain independence we actually need to consider the arm events from $k$ to $m-1$ and from $m+1$ to $n$). Let us now use \cite{GPS10} in order to prove the right-hand inequality:

Proposition 5.9 of \cite{GPS10} is the quasi-multiplicativity property when the set of resampled hexagons is an arbitrary deterministic set $W\subseteq\mathcal H_{k,n}$. Let $R_t$ denote the random set of resampled hexagons when defining $Y$ from $X$ (in particular, $R_t$ contains $\mathcal{H}_{k,n} \setminus S$). By conditioning on $R_t$ and applying \cite[Prop.\ 5.9]{GPS10} to $W=R_t$, we obtain that
\[
\E \left[ \E \left[ f^\star_{k,m}(X)f^\star_{k,m}(Y) \mid R_t \right] \E \left[ f^\star_{m,n}(X)f^\star_{m,n}(Y) \mid R_t \right] \right] \le C \Qu_t^S(f^\star_{k,n}).
\]
By using independence on disjoint sets, we conclude that the left-hand side is at least equal to $c\Qu_t^S(f^\star_{k,m}) \Qu_t^S(f^\star_{m,n})$ (the reason why a constant $c$ appears is the same as in the above paragraph). This concludes the first part of Lemma \ref{lem:qm}. Similarly, the second part of Lemma~\ref{lem:qm} follows from \cite[(5.10)]{GPS10}.
\end{proof}
Proposition~\ref{prop:qm} is the first part of Lemma~\ref{lem:qm} in the special case $S=\mathcal{H}_{k,n}$. Similarly, in the special case $S=\mathcal{H}_n$, the second part of Lemma~\ref{lem:qm} implies that, for $i$ as in the lemma,
\begin{equation}\label{eq:pivGPS}
\Qu_t(i \text{ is piv. for } g_n) \ge c \pi_n(t).
\end{equation}
\begin{remark}\label{rem:4.3}
By monotonicty of $t \mapsto \pi^\star_{n,2n}(t)$ (Proposition~\ref{prop:2}) and by the Russo--Seymour--Welsh theorem, we have $\pi^\star_{n,2n}(t) \ge \pi^\star_{n,2n}(1) = (\alpha_{n,2n}^\star)^2 \ge c$. As a result, Proposition~\ref{prop:qm} implies that $\pi_{2m,n/2}^\star(t)\le C\pi_{m,n}^\star(t)$. We will use several times this kind of inequality by only referring to Proposition~\ref{prop:qm}. More generally, Lemma~\ref{lem:qm} implies that, for every $S \subseteq \mathcal{H}_{m,n}$, $\Qu_t^S(f^\star_{2m,n/2}) \le C\Qu_t^S(f^\star_{m,n})$.
\end{remark}

\section{Differential inequalities}\label{sec:diff}

In this section, we prove some differential inequalities for the dynamical arm and crossing events which are reminiscent of those obtained by Kesten in \cite{Kes87}. Recall that $g_n$ is the indicator function of the left-right crossing of $[0,n]^2$.

\begin{lemma}\label{lem:diff_g_n}
There exist $c,C>0$ such that, for every $n \geq 1$ and every $t\in [0,1]$
\[
c n^2\pi_n(t)\le -\frac{d}{dt}\COV_t(g_n)\le  C\sum_{k=1}^n \sum_{r=k}^n \pi_k(t) \pi_{k,r}^{010+}(t) \pi_{r,n}^{01+}(t).
\]
\end{lemma}

\begin{proof}
Recall that $\mathcal{H}_n$ is the set of hexagons that intersect $[0,n]^2$. By applying Proposition~\ref{prop:1} to the increasing function $g_n \, : \, \{0,1\}^{\mathcal{H}_n} \rightarrow \{0,1\}$ (and to $S=\mathcal{H}_n$), we obtain that
\begin{equation}\label{eq:diff_g_n}
-\frac{d}{dt} \COV_t(g_n) = -\frac{d}{dt} \Qu_t(g_n) = \frac 1 4 \sum_{i\in \mathcal{H}_n} \Qu_t ( i \text{ is piv. for } g_n).
\end{equation}
Equations \eqref{eq:pivGPS} and \eqref{eq:diff_g_n} imply the lower bound of Lemma~\ref{lem:diff_g_n}. Let us prove the upper bound. Let $i \in \mathcal{H}_n$, let $i_0$ denote the closest point on $\partial [0,n]^2$ and let $i_1$ denote the closest corner of $[0,n]^2$. Let $k$ be the distance between $i$ and $i_0$ and $r$ be the distance between $i_0$ and $i_1$. If $i$ is pivotal, then there is a four-arm event centered at $i$ from $1$ to $k/4$, an arm event of type $010+$ or $101+$ centered at $i_0$ from $4k$ to $r/4$ and an arm event of type $01+$ or $10+$ centered at $i_1$ from $4r$ to $n$ (actually there even is such an arm event in the quarter-plane). See Figure~\ref{fig:diff_cross} (in this figure, as well as in all the figures of the present section, we have drawn arm events at only one time, but these events occur at two different times simultaneously). This observation, spatial independence, and then the quasi-multiplicativity property Proposition~\ref{prop:qm} imply that the right-hand side of \eqref{eq:diff_g_n} is at most
\[
C\sum_{k=1}^n \sum_{r=k}^n \pi_{k/4}(t) \pi_{4k,r/4}^{010+}(t) \pi_{4r,n}^{01+}(t) \le C\sum_{k=1}^n \sum_{r=k}^n \pi_{k}(t) \pi_{k,r}^{010+}(t) \pi_{r,n}^{01+}(t). \qedhere
\]
\end{proof}

\begin{figure}[h!]
\begin{center}
\includegraphics[height=12em]{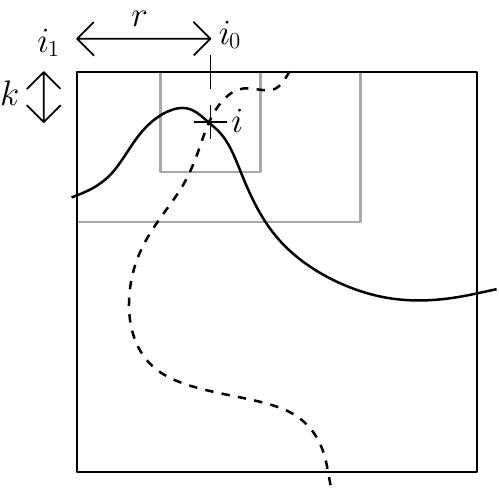}
\caption{\small The pivotal event in Lemma~\ref{lem:diff_g_n}.\label{fig:diff_cross}}
\end{center}
\end{figure}

As in \cite{Kes87}, we consider a modified version of the arm events probabilities in order to being able to ignore the effect of the points close the boundary. Recall that $\mathcal{H}_{m,n}$ is the set of hexagons that intersect the annulus $[-n,n]^2 \setminus (-m,m)^2$ and recall the notation $\Qu_t^S$ from the beginning of Section~\ref{sec:russo}. For any arm type $\star$, we let
\[
\widetilde{\pi}^\star_{m,n}(t) = \Qu_t^{\mathcal{H}_{2m,n/2}} (f^\star_{m,n}), \quad \widetilde{\pi}_{m,n}(t)=\widetilde{\pi}^{0101}_{m,n}(t).
\]
\begin{lemma}\label{lem:tilde}
For any arm type $\star$ (i.e.\ $\star=\sigma$ or $\sigma+$), there exists $c>0$ such that, for every $n \ge m \ge 1$ and every $t \in [0,1]$,
\[
c\pi^\star_{m,n}(t) \le \widetilde{\pi}^\star_{m,n}(t) \le \pi^\star_{m,n}(t).
\]
\end{lemma}
\begin{proof}
Proposition~\ref{prop:2} implies the second inequality. Moreover, by Lemma~\ref{lem:qm} (and Remark \ref{rem:4.3})
\[
\widetilde{\pi}^\star_{m,n}(t) = \Qu_t^{\mathcal{H}_{2m,n/2}}(f_{m,n}^\star) \ge c \Qu_t^{\mathcal{H}_{2m,n/2}}(f_{2m,n/2}^\star) = c \pi^\star_{2m,n/2}(t) \ge c \pi^\star_{m,n}(t). \qedhere
\]
\end{proof}

Let us now prove a differential inequality for the four-arm event. (As we will observe below, one can show this differential inequality -- with the same proof -- for any alternating arm event in the plane.)

\begin{lemma}[Differential inequalities for the four-arm event]\label{lem:diff_pi}
There exists $C>0$ such that, for every $n \ge m \ge 1$,
\[
0 \leq -\phi'\le C\phi \sum_{k=m}^n k\pi_k,
\]
where $\phi=\widetilde{\pi}_{m,n}$.
\end{lemma}

\begin{proof}
If we apply Proposition~\ref{prop:1} to $f=f_{m,n}^{0101}$ and $S=\mathcal{H}_{2m,n/2}$ we obtain that
\begin{multline}\label{eq:diff_alt}
0\leq-\frac{d}{dt}\widetilde{\pi}_{m,n}(t) \leq  \frac 1 4 \sum_{i\in \mathcal{H}_{2m,n/2}} \Qu_t^{\mathcal{H}_{2m,n/2}} ( i\text{ is piv. for } f_{m,n}^{0101} )\\
\overset{\text{Prop. \ref{prop:2}}}{\le} \frac 1 4  \sum_{i\in \mathcal{H}_{2m,n/2}} \Qu_t ( i\text{ is piv. for } f_{m,n}^{0101}).
\end{multline}
Let $i \in \mathcal{H}_{2m,n/2}$. If $i$ is at distance $k$ from the origin and is pivotal, then $f_{m,k/4}^{0101}=f_{4k,n}^{0101}=1$ and there is a four-arm event from $i$ to distance $k/4$ (here we use that $i \in \mathcal{H}_{2m,n/2}$; this would not be true for some $i$ close to the boundary of $\mathcal{H}_{m,n}$). See Figure~\ref{fig:diff_pi_al}. This observation, spatial independence, and then the quasi-multiplicativity property of  Proposition~\ref{prop:qm} imply that the right-hand side of \eqref{eq:diff_alt} is at most
\[
C\sum_{k=m}^{n} k \pi_{m,k/4}(t) \pi_{4k,n}(t) \pi_{k/4}(t) \le C\sum_{k=m}^n k \pi_{m,n}(t) \pi_k(t).
\]
We conclude the proof by applying Lemma~\ref{lem:tilde}.
\end{proof}

\begin{figure}[h!]
\begin{center}
\includegraphics[height=14em]{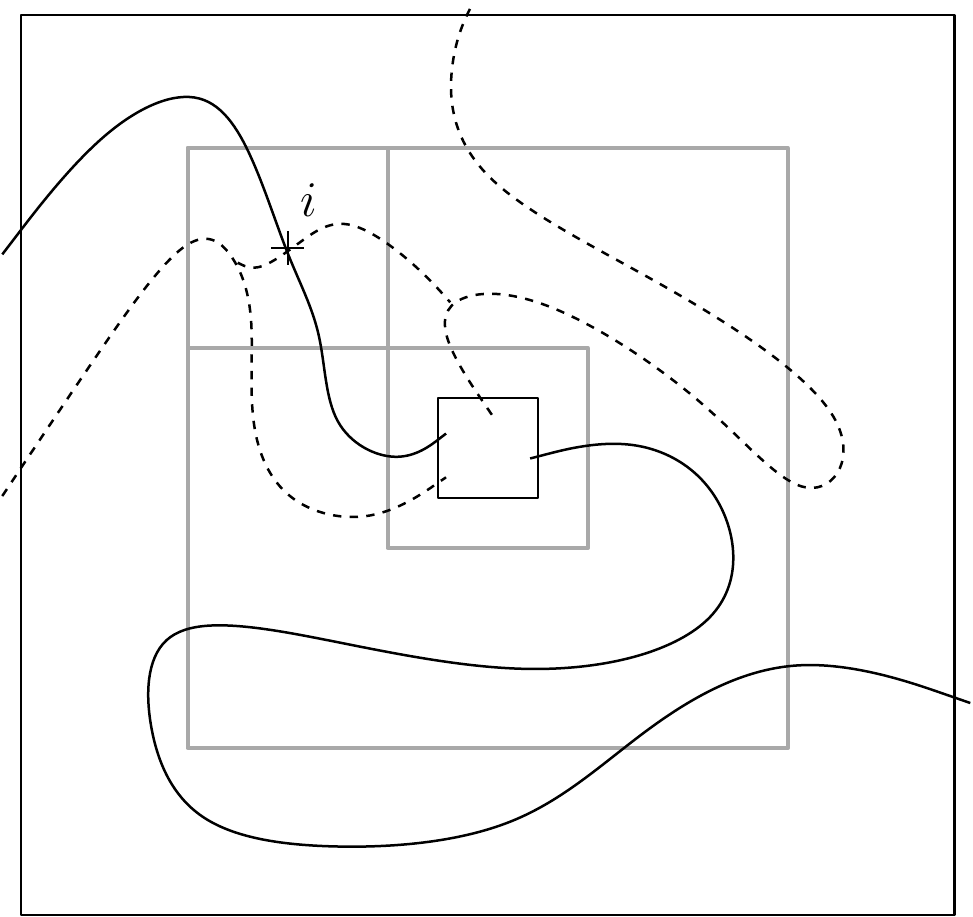}
\caption{\small The pivotal event in Lemma~\ref{lem:diff_pi}.\label{fig:diff_pi_al}}
\end{center}
\end{figure}

In the rest of this section (and in Section~\ref{sec:sharp-noise-sens}), we will rely on the computation of the universal arm events, see for instance Chapter 1 of \cite{Wer07}: Let $n \ge m \ge 1$. Then,
\begin{equation}\label{eq:39}
\alpha^{01,+}_{m,n} \asymp \frac{m}{n}, \quad \alpha^{010,+}_{m,n} \asymp \left( \frac{m}{n} \right)^2, \quad \alpha^{01001}_{m,n} \asymp \left( \frac{m}{n} \right)^2.
\end{equation}
Let us note that Reimer's inequality (see for instance \cite{Gri99}) and \eqref{eq:39} for the arm type $01001$ (together with the Russo--Seymour--Welsh theorem which implies that $\alpha_{m,n}^1 \le C(m/n)^c$) imply that
\begin{equation}\label{eq:sixarm}
\alpha^{001001}_{m,n} \le C \left( \frac{m}{n} \right)^{2+c}.
\end{equation}  
The next lemma is the analogue of Lemma~\ref{lem:diff_pi} for alternating arm events in the half-plane. 
\begin{lemma}[Differential inequalities for alternating half-plane arm events]
  \label{lem:1}
  Let $j \ge 1$ and let $\sigma \in \{0,1\}^j$ such that $\sigma_{i+1} \neq \sigma_i$ for every $1 \le i \le j-1$. There exists $C>0$ such that, for every $n \ge m \ge 1$,
  \begin{equation*}
  0\le- \phi'\le C \phi  \sum_{k=1}^n\frac {k^2 \pi_k}{\max(k,m)},
\end{equation*}
where $\phi=\widetilde{\pi}_{m,n}^{\sigma+}$.
\end{lemma}

\begin{proof} 
As in the proof of Lemma~\ref{lem:diff_pi}, we have
\begin{equation}\label{eq:diff_half}
0\leq-\frac{d}{dt}\widetilde{\pi}^{\sigma+}_{m,n}(t) \le \frac 1 4 \sum_{i\in \mathcal{H}_{2m,n/2}} \Qu_t ( i\text{ is piv. for } f_{m,n}^{\sigma+}).
\end{equation}
Let $i \in \mathcal{H}_{2m,n/2}$, let $i_0$ denote the closest point on the boundary of the half-plane, let $r$ be the distance between $i$ and the origin and let $k$ be the distance between $i$ and $i_0$. If $i$ is pivotal, then $f_{m,r/4}^{\sigma+}=f_{4r,n}^{\sigma+}=1$, there is a four-arm event centered at $i$ from $1$ to $k/4$ and there is an arm event of type $010+$ or $101+$ centered at $i_0$ from $4k$ to $r/4$. See Figure~\ref{fig:diff_pi_half}. This implies that the right-hand side of \eqref{eq:diff_half} is at most
\[
C  \sum_{r=m}^{n} \sum_{k=1}^r \pi^{\sigma+}_{m,n}(t) \pi_k(t) \pi_{k,r}^{010+}(t).
\]
By~\eqref{eq:39} and by monotonicity of $t \mapsto \pi_{k,r}^{010+}(t)$ (Proposition~\ref{prop:2}), we have $\pi_{k,r}^{010+}(t) \le \pi_{k,r}^{010+}(0)=\alpha_{k,r}^{010+} \le C (k/r)^2$. As a result,
\begin{multline*}
-\frac{d}{dt}\widetilde{\pi}^{\sigma+}_{m,n}(t) \le C\sum_{r=m}^{n} \sum_{k=1}^r \pi^{\sigma+}_{m,n}(t) \pi_k(t) (k/r)^2\\
= C\sum_{k=1}^{n} \pi^{\sigma+}_{m,n}(t) k^2\pi_k(t) \sum_{r=\max(k,m)}^n \frac{1}{r^2} \le C \sum_{k=1}^n \pi^{\sigma+}_{m,n}(t)\frac {k^2 \pi_k(t)}{\max(k,m)}.
\end{multline*}
We conclude the proof by applying Lemma~\ref{lem:tilde}.
\end{proof}

\begin{figure}[h!]
\begin{center}
\includegraphics[height=10em]{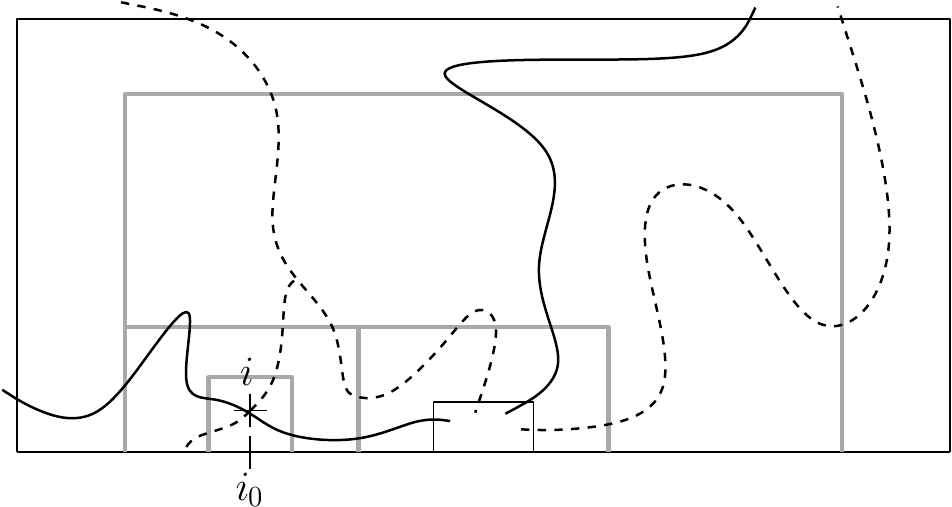}
\caption{\small The pivotal event in Lemma~\ref{lem:1}.\label{fig:diff_pi_half}}
\end{center}
\end{figure}

Let us now prove a differential inequality for general arm events. 

\begin{lemma}[Differential inequalities for general arm events]
  \label{lem:2}
Let $\star$ be any arm type (we recall that in all the paper this means that $\star=\sigma$ or $\sigma+$). There exist $c,C>0$ such that for every $n \ge m \ge 1$,
  \begin{equation}\label{eq:1}
  0\le- \phi'\le C \phi  \sum_{k=1}^n\frac {k^{1+c} \pi_k}{\max(k,m)^c},
\end{equation}
where $\phi=\widetilde{\pi}_{m,n}^{\star}$.
\end{lemma}

\begin{proof}
Let us first observe that Lemma~\ref{lem:diff_pi} gives the desired result if $\star = 0101$ (and in this case we can restrict the sum to $m \le k \le n$). More generally, one can show - with exactly the same proof - that the formula of Lemma~\ref{lem:diff_pi} holds for $\star=\sigma$ for any alternating $\sigma \in \{0,1\}^j$ (i.e.\ $\sigma_{i+1} \neq \sigma_i$ for every $i \in \{1,\cdots,j\}$ where $\sigma_{j+1}:=\sigma_1$). Let us also note that Lemma~\ref{lem:1} gives the desired result (with $c=1$) if $\star=\sigma+$ for some alternating $\sigma \in \{0,1\}^j$ (i.e.\ $\sigma_{i+1} \neq \sigma_i$ for every $i \in \{1,\cdots,j-1\}$). Hence it remains to prove \eqref{eq:1} for non-alternating arm events. 
\medskip

Let us start with non-alternating arm events in the plane, i.e.\ for $\phi=\widetilde{\pi}_{m,n}^{\sigma}$ for some non-alternating $\sigma \in \{0,1\}^j$. As in the proof of Lemma~\ref{lem:diff_pi}, we have
\begin{equation}\label{eq:diff_gen}
0\leq-\frac{d}{dt}\widetilde{\pi}^{\sigma}_{m,n}(t) \le \frac 1 4 \sum_{i\in \mathcal{H}_{2m,n/2}} \Qu_t ( i\text{ is piv. for } f_{m,n}^\sigma).
\end{equation}
Let $i \in \mathcal{H}_{2m,n/2}$. We say that an event $A \subseteq \{0,1\}^{\mathcal{H}_{m,n}}$ holds with at most $d$ defaults for a configuration $x \in \{0,1\}^{\mathcal{H}_{m,n}}$ if there exists $y \in A$ that differs from $x$ in at most $d$ coordinates. If $i$ is at distance $r$ from the origin and is pivotal, then $f_{m,r/4}^{\sigma}=f_{4r,n}^{\sigma}=1$ and there exists $1 \le k \le \log_4(r)$ such that the following holds: there is a four-arm event centered at $i$ from $1$ to $4^{k-1}$ and there is an arm event of type $001001$ or $110110$ with at most $j-2$ defaults centered at $i$ from $4^{k+1}$ to $r/4$. See Figure~\ref{fig:diff_pi_gen} (we refer to \cite{Nol08} -- e.g.\ Figure 12 therein -- for similar observations).

\begin{figure}[h!]
\begin{center}
\includegraphics[height=18em]{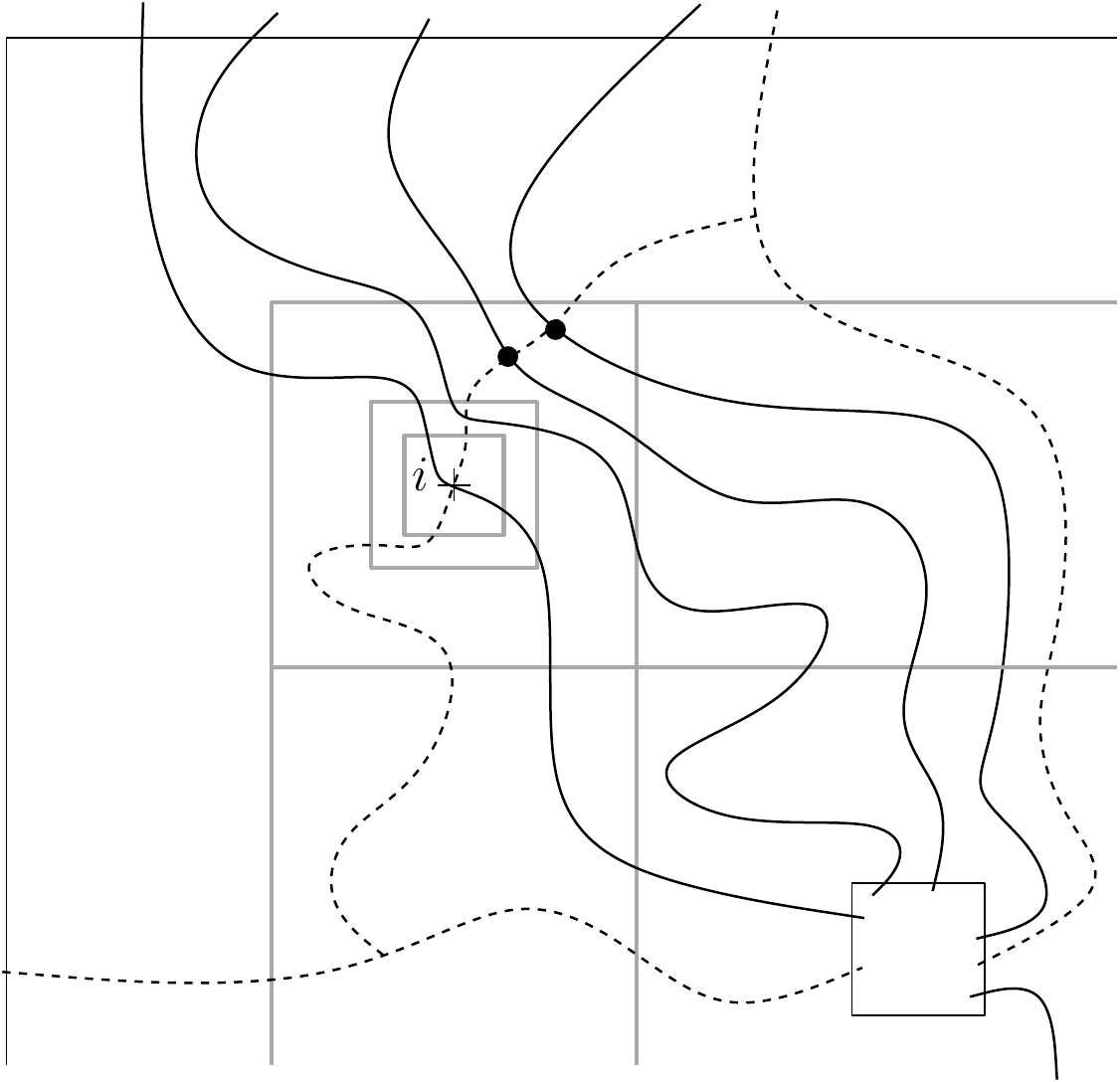}
\caption{\small The pivotal event in Lemma~\ref{lem:2}. There is a four-arm event centered at $i$ from $1$ to $4^{k-1}$ and an arm event of type $001001$ centered at $i$ from $4^{k+1}$ to $r/4$ with two defaults.\label{fig:diff_pi_gen}}
\end{center}
\end{figure}

Let $D_{m,n}$ denote the arm event of type $001001$ from $m$ to $n$ with at most $j-2$ defaults. By the quasi-multiplicativity property\footnote{Let us be more precise about the proof. One first defines a family of dyadic annuli as follows: $\mathcal{A}:=\{A_r,\dots,A_R\}$ where $A_i=[-2^{i},2^i]^2 \setminus (-2^{i-1},2^{i-1})^2$, $r=\lfloor \log_2(m) \rfloor$ and $R=\lceil \log_2(n) \rceil$. One then notes that $D_{m,n}$ is included in the union on every $(j-2)$-tuple $(A_{i_1},\dots,A_{i_{j-2}}) \in \mathcal{A}^{j-2}$ of the event that the $6$-arm event holds in $[-n,n]^2 \setminus (-m,m)^2$ minus $A_{i_1}\cup \dots \cup A_{i_{j-2}}$. One then uses the quasi-multiplicativity property $j-2$ times (this is the reason why the term $C^{j-2}$ appears) for each choice of $(j-2)$-tuple. Finally, one can conclude by using the union bound (which is the reason why the term $(1 + \log( n/m ))^{j-2}$ appears). See e.g.\ \cite[Proposition 17]{Nol08} for similar arguments.} of Proposition~\ref{prop:qm},
\[
\Qu_t(D_{m,n}) \le \left( C \left( 1 + \log( n/m ) \right) \right)^{j-2} \pi_{m,n}^{001001}(t). 
\]
By \eqref{eq:sixarm} and by monotonicity, $\pi_{m,n}^{001001}(t) \le \pi_{m,n}^{001001}(0)=\alpha_{m,n}^{001001} \le C (m/n)^{2+c}$. As a result,
\[
\Qu_t(D_{m,n}) \le \big( C ( 1 + \log( n/m ) ) \big)^{j-2} \left( \frac{m}{n} \right)^{2+c} \le C \left( \frac{m}{n} \right)^{2+c}.
\]
(Recall that the constants $c$ and $C$ are not necessarily the same at each occurrence, even within the same formula.) As in the proof of Lemma~\ref{lem:1}, these observations imply that the right-hand side of \eqref{eq:diff_gen} is at most
\[
C\sum_{r=m}^{n} r \sum_{k=1}^{\log_4(r)} \pi^{\sigma}_{m,n}(t) \pi_{4^k}(t) \left( \frac{4^k}{r} \right)^{2+c} \le C \sum_{r=m}^{n} r \sum_{k'=1}^{r} \frac{1}{k'} \pi^{\sigma}_{m,n}(t) \pi_{k'}(t) \left( \frac{k'}{r} \right)^{2+c}.
\]
We end the proof by permuting the sums and then applying Lemma~\ref{lem:tilde}.

\medskip
It only remains to prove the result for non-alternating arm events in the half-plane, i.e.\ for $\phi=\widetilde{\pi}_{m,n}^{\sigma+}$ for some non-alternating $\sigma$. Let $i \in \mathcal{H}_{2m,n/2}$, let $i_0$ denote the closest point on the boundary of the half-plane, let $r$ be the distance between $i$ and the origin and let $k$ be the distance between $i$ and $i_0$. If $i$ is pivotal then $f_{m,r/4}^{\sigma}=f_{4r,n}^{\sigma}=1$ and one of the following two unions of arm events hold in annuli at larger and larger scales centered at $i$ or $i_0$: either (1) a four-arm event, an event of type $001001$ (or $110110$) with defaults and an event of type $0010+$ (or $1101+$) with defaults, or (2) a four-arm event, an event of type $010+$ (or $101+$) and an event of type $0010+$ (or $1101+$) with defaults. Since the proof is essentially a combination of the proof for non-alternating arm events in the plane and of the proof of Lemma~\ref{lem:1} (and since this specific result is only used to obtain Theorem~\ref{thm:arm} for non-alternating arm events in the half-plane), we leave the details to the reader.
\end{proof}

\section{Superquadratic decay above the sensitivity length}\label{sec:new-proof-sharp}

In this section, we prove the following bound for the dynamical four-arm event, which is the central proposition of our paper. Recall that $\ell(t)$ is the sensitivity length (defined in \eqref{eq:sens_length}).

\begin{proposition}[Superquadratic behaviour above the sensitivity length]\label{prop:super_quadratic}
There exist $c,C>0$ such that, for every $t \in [0,1]$ and every $n \ge m \ge \ell(t)$,
  \begin{equation}
    \label{eq:11}
    \frac{\pi_n(t)}{\pi_m(t)}\le C\left(\frac m n\right)^{2+c}.
  \end{equation}
\end{proposition}

The proof of the proposition above  will involve only elementary analysis arguments, and rely on properties of the two variable-function $(t,n)\mapsto\pi_n(t)$ that have been established in the previous sections. Before starting the proof, we gather the three main properties \textbf{(\ref{eq:12}-\ref{eq:14})} that we will use.  

As mentioned earlier, our first main ingredient is the differential inequality satisfied by the dynamical four-arm probability from Lemma~\ref{lem:diff_pi}. By integrating this differential inequality, and then  using Lemma~\ref{lem:tilde} and the quasi-multiplicativity property~\eqref{eq:38fractional}, we obtain the following integral inequality: For every $n\ge 1$ and every  $0\le t\le u\le 1$,
\begin{equation}
  \label{eq:12}
  \tag{\textbf{P1}} 1\le \frac{\pi_n(t)}{\pi_n(u)}\le C \exp\left(C\int_t^u\sum_{k=1}^nk\pi_k(s)ds\right).  
\end{equation}
Second, using that $cn^2\pi_n$ is a lower bound on the derivative with respect to $t$ of $\COV_t(g_n)$ where $g_n$ is the indicator of the crossing event (see Lemma~\ref{lem:diff_g_n}), and integrating for $t$ from $0$ to $1$, we  get an other integral inequality: For every $n\ge 1$,
\begin{equation}
  \label{eq:13}
  \int_0^1 n^2 \pi_n(t) dt \le C. \tag{\textbf{P2}}
\end{equation}
Third, using monotonicity of $t \mapsto \pi_{m,n}(t)$ (Proposition~\ref{prop:2}) and quasi-multiplicativity  (Proposition~\ref{prop:qm}), we obtain the following monotonicity property. For every $n\ge m\ge 1$ and every  $0\le t\le u\le 1$,
\begin{equation}
  \label{eq:14}
\pi_n(u)\pi_m(t)\le C \pi_m(u)\pi_n(t).  \tag{\textbf{P3}}
\end{equation}
This last property can be thought as a ``quasi-monotonicity'' property  of the ratio $\pi_m/\pi_n$ for fixed $m\le n$  (in the sense that $t\le u$ implies $C\pi_n(t)/\pi_m(t) \ge  \pi_n(u)/\pi_m(u)$) or, equivalently,  as a quasi-monotonicity in $n$ of the ratio $\pi_n(t)/\pi_n(u)$ for fixed $t\le u$.

We list below the properties about the static four-arm event that we use in this section. We first recall that $\varepsilon_n = \tfrac{1}{n^2\alpha_n}$, and therefore, by \eqref{eq:fourarm_lower_back}, we have  
\begin{equation}\label{eq:u_m/u_n}
\frac{\varepsilon_m}{\varepsilon_n} \ge c\left( \frac{n}{m} \right)^c.
\end{equation}
The two functions $n \in \N^* \mapsto \varepsilon_n$ and $t \in [0,1] \mapsto \ell(t)$ can be thought as reciprocal decreasing bijections even if this is not exactly true (note that $t \mapsto \ell(t)$ is clearly  non-increasing, while $\varepsilon_n$ has no reason to be non-increasing in $n$). The following weaker statements, which can be deduced from the definitions and \eqref{eq:u_m/u_n}, will be sufficient for the proof:
\begin{equation}\label{eq:bij}
ct \le \varepsilon_{\ell(t)} = \frac{1}{\ell(t)^2\alpha_{\ell(t)}} \le t;
\end{equation}
\begin{equation}
  \label{eq:23}
  n\ge C_1 m \implies  \varepsilon_n \le  \varepsilon_m,
\end{equation}
where $C_1$ is some positive constant that we fix until the end of the section. (The first inequality from \eqref{eq:bij} comes from the fact that $\varepsilon_{\ell(t)} \ge c\varepsilon_{\ell(t)-1} \ge ct$, using $\alpha_{\ell(t)}\le \alpha_{\ell(t)-1}$, the second inequality holds by definition of $\ell(t)$, and \eqref{eq:u_m/u_n} implies \eqref{eq:23}.)

\begin{proof}[Proof of Proposition~\ref{prop:super_quadratic}]
The proof is decomposed into three steps. First, we use \eqref{eq:12} to   show that  the dynamical four-arm probability $\pi_n(t)$ behaves like $\alpha_n$ under the sensitivity length (i.e. for $n\le \ell(t)$), and then we use this estimate together with (\ref{eq:13}-\ref{eq:14}) and the polynomial lower bound \eqref{eq:u_m/u_n} to conclude the proof in two steps: first a logarithmic, then the promised polynomial improvement.

\medskip 
\noindent \textbf{Step 1: Stability of $\pi_n(t)$  below the sensitivity length.}

\smallskip
\noindent By monotonicity, we  have $\pi_n(s)\le \pi_n(0)=\alpha_n$ for every $s \in [0,1]$. Therefore, for every $t\in[0,1]$ and every $n \ge 1$, we have 
\begin{align*}
 \int_0^t \sum_{k=1}^nk \pi_k(s)ds \leq t\sum_{k=1}^nk\alpha_k  \overset{~\eqref{eq:fourarm_lower_back}}{\le} C t n^2\alpha_n =  C\frac t {\varepsilon_n}.
\end{align*}
The above is less than a constant if $t \leq \varepsilon_n$ or if $n \leq \ell(t)$ (indeed, if $n < \ell(t)$ then $t/\varepsilon_n < 1$ and if $n = \ell(t)$ then $t/\varepsilon_n \le C$ by \eqref{eq:bij}). Hence, Property \eqref{eq:12} implies that for every $t\in[0,1]$ and every $n \ge 1$ such that $n \le \ell(t)$ or $t \le \varepsilon_n$, we have
\begin{align}
  \label{eq:15}
  \alpha_n\ge \pi_n(t)&\ge c\alpha_n.
\end{align}

\medskip
\noindent \textbf{Step 2: Improvement by a log-factor.}
\smallskip

\noindent In this step, we prove that there exists a constant $C_2>0$ such that for every $t\in [0,1]$, $m\ge C_2\ell(t)$ and $n\ge C_2 m$,
\begin{equation}
  \label{eq:22}
  \frac{\pi_n(t)}{\pi_m(t)}\le \frac{C_2}{\log(n/m)} \cdot \left( \frac m n \right)^2.
\end{equation}
Let us fix $t\in [0,1]$. It follows from \eqref{eq:23} and \eqref{eq:bij} that $\varepsilon_m \le t$ for every $m\ge C_1 \ell(t)$. Hence by \eqref{eq:14}, we have for every $n\ge m\ge C_1 \ell(t)$,
\begin{equation}
  \label{eq:24}
  \frac{\pi_n(t)}{\pi_m(t)}\le C \frac{\pi_n(\varepsilon_m)}{\pi_m(\varepsilon_m)}\overset{\eqref{eq:15}}\le  C \frac{\pi_n(\varepsilon_m)}{\alpha_m},
\end{equation}
and by~\eqref{eq:13},
\begin{equation}
  \label{eq:25}
  \pi_n(\varepsilon_m)= \frac{\int_0^{\varepsilon_m} \pi_n(s) ds}{\int_0^{{\varepsilon_m}} \frac{\pi_n(s)}{\pi_n(\varepsilon_m)}  ds}\le\frac C{n^2} \bigg(\int_0^{{\varepsilon_m}} \frac{\pi_n(s)}{\pi_n(\varepsilon_m)}  ds\bigg)^{-1}.
\end{equation}
Plugging this estimate into \eqref{eq:24}, we obtain that, for every $n\ge m\ge C_1\ell(t)$,
\begin{equation}\label{eq:27}
\frac{\pi_n(t)}{\pi_m(t)}\le C\left(\frac m n\right)^2\cdot \left(\frac 1{\varepsilon_m} \int_0^{{\varepsilon_m}} \frac{\pi_n(s)}{\pi_n(\varepsilon_m)}  ds\right)^{-1}.
\end{equation}
Notice that we already get a quadratic upper bound by using $\pi_n(s)/\pi_n(\varepsilon_m)\ge 1$ in the integral. In order to get the desired logarithmic improvement \eqref{eq:22}, we use the following improved lower bound. 
Let  $m\ge C_1 \ell(t)$ and $n\ge C_1 m$. Notice that $\varepsilon_n\le\varepsilon_m$ by \eqref{eq:23}. For $s\in [\varepsilon_n,\varepsilon_m]$, we have $n\ge \ell(s)$ and therefore \eqref{eq:14} implies that
\begin{equation*}
\frac{\pi_n(s)}{\pi_n(\varepsilon_m)}\ge c \frac{\pi_{\ell(s)}(s)}{\pi_{\ell(s)}(\varepsilon_m)} \overset{\eqref{eq:15}}\ge c\frac{\alpha_{\ell(s)}}{\pi_{\ell(s)}(\varepsilon_m)}.
\end{equation*}
Using  $\pi_{\ell(s)}(\varepsilon_m)\le C/(\ell(s)^2\varepsilon_m)$ (which follows from \eqref{eq:25} and the monotonicity bound $\pi_n(s)/\pi_n(\varepsilon_m)\ge 1$), we finally get for every  $s\in [\varepsilon_n,\varepsilon_m]$
\begin{equation*}
  \frac{\pi_n(s)}{\pi_n(\varepsilon_m)}\ge  c \frac{\varepsilon_m}{\varepsilon_{\ell(s)}} \overset{\eqref{eq:bij}}{\ge} c\frac{\varepsilon_m}{s}.
\end{equation*}
 Therefore, by integration,
\begin{equation*}
\frac 1{\varepsilon_m}\int_0^{{\varepsilon_m}} \frac{\pi_n(s)}{\pi_n(\varepsilon_m)}ds \ge c \int_{\varepsilon_n}^{{\varepsilon_m}} \frac{ds}{s}=c\log(\varepsilon_m/\varepsilon_n).
\end{equation*}
Plugging this estimate in~\eqref{eq:27} and using the polynomial lower bound~\eqref{eq:u_m/u_n} concludes the proof of \eqref{eq:22}.

\medskip
\noindent \textbf{Step 3: Scale by  scale decomposition.}

\smallskip
\noindent Let $C_2>0$ as in \eqref{eq:22}, and fix $\lambda\ge C_2$. Let $t\in [0,1]$ and $n\ge m\ge C_1 \ell(t)$ and  assume for simplicity that $n=\lambda^km$ for some integer $k\ge1$ (the result can  easily be extended to more general $n\ge m\ge\ell(t)$ by quasi-multiplicativity). By decomposing the ratio ${\pi_{n}(t)}/{\pi_{m}(t)}$ into a telescopic product and applying \eqref{eq:22} to each term of this product we obtain
\begin{equation*}
\frac{\pi_{n}(t)}{\pi_{m}(t)}=\prod_{i=1}^k \frac{\pi_{\lambda^im}(t)}{\pi_{\lambda^{i-1}m}(t)} \overset{\eqref{eq:22}} \le \left(\frac m n\right)^2 \left(\frac{C_2}{\log \lambda}\right)^k\le \left(\frac m n\right)^{2+c},
\end{equation*}
provided $\lambda$ is chosen large enough. This completes the proof of~\eqref{eq:11}.
\end{proof}

Additionally, we notice that our approach gives a new proof of the following upper bound for the (multiscale) static four-arm event proven by Garban in Appendix B of \cite{SS11}. However, our approach does not imply the more quantitative result $\alpha_n/\alpha_m \le C(m/n) \sqrt{\alpha_n^{01}/\alpha_m^{01}}$ that can be extracted from \cite[Appendix B]{SS11}. See also \cite{vdBN20} for different proofs in the special case $m=1$.
\begin{corollary}
  \label{cor:2}
  There exist $c,C>0$ such that for every $n\ge m\ge 1$
  \begin{equation*}
    \frac{\alpha_n}{\alpha_m}\le C\left(\frac m n\right)^{1+c}.
  \end{equation*}
\end{corollary}
\begin{proof}
This is Proposition~\ref{prop:super_quadratic} in the special case $t=1$ (by taking a square root).
\end{proof}

\section{Sharp noise sensitivity for arm  events}\label{sec:sharp}

In this section, we use Proposition~\ref{prop:super_quadratic} to prove Theorem~\ref{thm:arm}. Our proof  relies on the differential inequality of Lemma~\ref{lem:2} which, together with Proposition \ref{prop:super_quadratic}, allows us to prove that the dynamical arm probability $\pi_{m,n}^\star$ ``does not vary'' when restricted to the stability region (i.e.\ $m\le n\le\ell(t)$) or the sensitivity region (i.e.\ $n\ge m \ge\ell(t)$).

Let $\star$ be an arm type.  First integrating the differential inequality of Lemma~\ref{lem:2}, and then using Lemma~\ref{lem:tilde}, we get that for every $n \ge m \ge 1$ and every $0 \le t \le u \le 1$,
\begin{equation}
  \label{eq:exp_star}
1\le \frac{\pi_{m,n}^\star(t)}{\pi_{m,n}^\star(u)}\le C \exp\left(C\int_t^u\sum_{k=1}^n \frac{k^{1+c}}{\max(k,m)^c}\pi_k(s)ds\right).  
\end{equation}
\begin{proof}[Proof of Theorem~\ref{thm:arm}, Item 1]  Let $m,n,t$ such that $\ell(t) \ge n \ge m \ge 1$. By \eqref{eq:exp_star}, it is sufficient to prove that
\[
\int_0^t \sum_{k=1}^n \frac{k^{1+c}}{\max(k,m)^c} \pi_k(s) ds \le \int_0^t \sum_{k=1}^n k \pi_k(s) ds
\]
is bounded from above by a constant. This is proven in Step 1 of the proof of Proposition~\ref{prop:super_quadratic}.
\end{proof}

\begin{proof}[Proof of Theorem~\ref{thm:arm}, Item 2]
Let $m,n,t$ such that $n\ge m \ge \ell(t)$. By \eqref{eq:exp_star}, it is sufficient to prove that
\[
\int_t^1 \sum_{k=1}^n \frac{k^{1+c}}{\max(k,m)^c} \pi_k(s) ds
\]
is bounded from above by a constant. Let us first deal with the sum restricted to $k \le m$. We have
  \begin{equation*}
\int_t^1 \sum_{k=1}^m \frac{k^{1+c}}{m^c} \pi_k(s) ds = \sum_{k=1}^m \frac{k^{c-1}}{m^c} \underbrace{\int_t^1 k^2 \pi_k(s) ds}_{\leq C \text{ by \eqref{eq:13}}} \leq C.
  \end{equation*}
Let us now deal with the sum restricted to $m \le k \le n$. Since $\ell$ is non-increasing, for every $s \in [t,1]$ we have $m \ge \ell(s)$, hence Proposition~\ref{prop:super_quadratic} applied to $k$, $m$ and $s$ implies that
  \begin{equation*}
\int_t^1\sum_{k=m}^n k \pi_k(s)ds \le C \int_t^1 \sum_{k=m}^n k \pi_m(s) \left( \frac{m}{k} \right)^{2+c} ds \le C \int_t^1 m^2 \pi_m(s) ds \overset{\eqref{eq:13}}\le C.
\end{equation*}
This ends the proof of Theorem~\ref{thm:arm}.
\end{proof}

\section{Sharp noise sensitivity for crossing events}
\label{sec:sharp-noise-sens}

In this section, we prove Theorem~\ref{thm:GPS}. Let us first prove the easiest, that is, the second point of the theorem.
\begin{proof}[Proof of Theorem~\ref{thm:GPS}, second point]
By Proposition~\ref{prop:1},
\begin{multline*}
0 \le \left(\Ex [g_n^2] - \Ex[g_n]^2 \right) - \COV_t(g_n) =  \Qu_0(g_n) - \Qu_t(g_n)\\ \le \int_0^t \sum_{i \in \mathcal{H}_n} \Qu_t(i \text{ is piv. for } g_n) dt
\overset{\text{Prop. \ref{prop:2}}}{\le} t \sum_{i \in \mathcal{H}_n} \Proba [i \text{ is piv. for } g_n ].
\end{multline*} 
We conclude the proof by using that $\sum_{i \in \mathcal{H}_n} \Proba [i \text{ is piv. for } g_n ] \asymp 1/\varepsilon_n$ (see for instance \cite[Section 6.5]{Wer07} or \cite[Proposition 6.8]{GS14}).
\end{proof}

The proof of the first point of Theorem \ref{thm:GPS} uses that $\COV_t(g_n)$ can be expressed as an integral, and we first start with a lemma that computes this integral.

\begin{lemma}
  For every $t\in [0,1]$ and every $n\ge \ell(t)$, we have 
  \label{lem:3}
  \begin{equation*}
    \int_t^1n^2\pi_n(s)ds \asymp  (1-t) \left( \frac{n \alpha_n}{\ell(t)\alpha_{\ell(t)}} \right)^2.
  \end{equation*}
\end{lemma}

\begin{proof}
We first observe that, if $n \ge 1$ and $s \in [0,1]$ are such that $\ell(s) \le n$ then (since by quasi-multiplicativity we have $\pi_n(s) \asymp \pi_{\ell(s)}(s) \pi_{\ell(s),n}(s)$) Theorem \ref{thm:arm} implies that
\begin{equation}\label{eq:etoile**}
\pi_n(s) \asymp \frac{(\alpha_n)^2}{\alpha_{\ell(s)}}.
\end{equation}

Let $t\in [0,1]$ and $n\ge \ell(t)$. For every $s \in [t,1]$, we have $\ell(s) \le \ell(t) \le n$. As a result, \eqref{eq:etoile**} implies that
  \begin{equation}\label{eq:int_lem8.1}
    \int_t^1n^2\pi_n(s)ds\asymp  n^2 \alpha_n^2\int_t^1\frac{ds}{\alpha_{\ell(s)}}\overset{\eqref{eq:bij}}\asymp n^2 \alpha_n^2\int_t^1 s \ell(s)^2ds.
  \end{equation}
 By using \eqref{eq:bij} again, Corollary~\ref{cor:2} and \eqref{eq:fourarm_lower_back}, we can show the two polynomial bounds for $1 \ge s \ge t \ge 0$
\begin{equation}\label{eq:poly_bounds_ell(u)}
c\left( \frac{t}{s} \right)^{1+c} \leq \frac{\ell(s)}{\ell(t)} \leq C\left( \frac{t}{s} \right)^{C}.
\end{equation}
This allows us  to treat $\ell(s)$ in the integral \eqref{eq:int_lem8.1} as if it was of the form  $s^\kappa$ for some $\kappa>1$. More precisely, this implies that
  \begin{equation*}
    \int_t^1 s \ell(s)^2ds \asymp (1-t)t^2\ell(t)^2 \overset{\eqref{eq:bij}}{\asymp}\frac{1-t}{\ell(t)^2\alpha_{\ell(t)}^2}. \qedhere
  \end{equation*}
\end{proof}

\begin{proof}[Proof of Theorem~\ref{thm:GPS}, first point]
We prove the following stronger statement: For every $t \in [0,1]$ and $n \ge \ell(t)$,
  \begin{equation}\label{eq:3}
     \COV_t(g_n)\asymp  (1-t) \left(\frac{n\alpha_n}{\ell(t) \alpha_{\ell(t)}} \right)^2
\end{equation}
(that is also proven in \cite{GPS10}). Let us first explain why \eqref{eq:3} implies the first point of Theorem~\ref{thm:GPS}. By Corollary~\ref{cor:2}, it is enough to show that $\ell(t_n)/n$ goes to $0$ if $t_n / \varepsilon_n$ goes to $+\infty$. To prove this, one can apply \eqref{eq:poly_bounds_ell(u)} to $t=\varepsilon_n$ and $s=t_n$ and use that $\ell(\varepsilon_n) \le n$.
 
Let us now prove \eqref{eq:3}. By Lemmas~\ref{lem:diff_g_n} and~\ref{lem:3}, it suffices to prove that, for every $s \in [t,1]$,
  \begin{equation}
    \label{eq:36}
    \sum_{k=1}^n \sum_{r=k}^n \pi_k(s) \pi_{k,r}^{010+}(s)  \pi_{r,n}^{01+}(s)\le C n^2\pi_n(s)
  \end{equation}
  (indeed, \eqref{eq:36} and Lemma~\ref{lem:diff_g_n} imply that $-\frac{d}{ds}\COV_s(g_n) \asymp n^2\pi_n(s)$). By Theorem~\ref{thm:arm} applied to $\star=1010$ and $\star = 010+$ and by the lower bound of \eqref{eq:fourarm_lower_back} as well as \eqref{eq:39} to estimate the arm events of type $010+$, we have $\pi_k(s) \pi_{k,r}^{010+}(s) \le C \pi_r(s)$ for every $r \ge k \ge 1$. As a result, for every $s \in [t,1]$,
\begin{equation}
\sum_{k=1}^n \sum_{r=k}^n \pi_k(s) \pi_{k,r}^{010+}(s) \pi_{r,n}^{01+}(s) \le C\sum_{r=1}^n r \pi_r(s) \pi_{r,n}^{01+}(s).\label{eq:37}
\end{equation}
Let us first deal with the sum for $r\le\ell(s)$. By Theorem~\ref{thm:arm} and quasi-multiplicativity applied to $\star=0101$ and $\star = 01+$, and by \eqref{eq:39} to estimate the arm events of type $01+$, we have
\begin{equation*}
  \sum_{r=1}^{\ell(s)} r \pi_r(s) \pi_{r,n}^{01+}(s)\le C \frac{\ell(s)}{n^2} \sum_{r=1}^{\ell(s)}  r^2 \alpha_r\overset{\eqref{eq:fourarm_lower_back}}{\le} C\frac{\ell(s)^4 \alpha_{\ell(s)}}{n^2}\overset{\eqref{eq:fourarm_lower_back}}{\le} C\frac{n^2\alpha_n^2}{\alpha_{\ell(s)}} \overset{\eqref{eq:etoile**}}{\le} C n^2\pi_n(s).
\end{equation*}
Similarly, for the  sum over  $r\ge\ell(s)$, we have 
\begin{equation*}
 \sum_{r=\ell(s)}^n r \pi_r(s) \pi_{r,n}^{01+}(s) \overset{\eqref{eq:etoile**}}{\le} C \frac1 {\alpha_{\ell(s)}n^2}\sum_{r=\ell(s)}^n r^3 \alpha_r^2 \overset{\eqref{eq:fourarm_lower_back}}{\le} C\frac{n^2\alpha_n^2}{\alpha_{\ell(s)}} \overset{\eqref{eq:etoile**}}{\le} C n^2\pi_n(s) .
\end{equation*}
(As observed in \cite{GPS10}, the boundary effects are actually negligible and the main contribution is controlled by the last displayed equation.) Plugging these two estimates in \eqref{eq:37} concludes the proof of \eqref{eq:36}, which ends the proof of the theorem.
\end{proof}

\section{Applications to dynamical percolation}\label{sec:dyna}

\subsection{Background on critical exponents}\label{ssec:exp}

Let $j \ge 1$ and let $\sigma \in \{0,1\}^j$ be a sequence that is not monochromatic. Lawler, Schramm and Werner \cite{LSW02} and Smirnov and Werner \cite{SW01} have proven that
\[
\alpha^\sigma_{m,n}=(m/n)^{\zeta_j+o(1)}, \quad \alpha^{\sigma+}_{m,n}=(m/n)^{\zeta_j^+ +o(1)},
\]
where $o(1)$ goes to $0$ as $m/n$ goes to $0$ and where
\[
\zeta_1 = 5/48; \quad \zeta_j = (j^2-1)/12 \text{ for every } j\geq 2; \quad \zeta_j^+ = (j^2+j)/6 \text{ for every } j \ge 1.
\]
Note that the Hausdorff dimensions from Theorem~\ref{thm:Hausdorff} are $(2-\zeta_4-\zeta^\star)/(2-\zeta_4)$, where $\zeta^\star$ equals $\zeta_1$, $\zeta_1^+$, $\zeta_2$ and $\zeta_3$ respectively. Moreover, the exponents from Remark~\ref{rem:exp_dyna_jge4} are $2-\zeta_4-\zeta^\star$.

\subsection{Computation of the Hausdorff dimensions}

\begin{proof}[Proof of Theorem~\ref{thm:Hausdorff}]
The upper bounds on these four Hausdorff dimensions are proved in \cite{SS10} (without using any spectral tool). We prove the lower bounds. Let $\star \in \{1,01,010,1+\}$. As explained in \cite[Section 6]{SS10}, by the Cauchy--Schwarz inequality and Frostman's theorem, it is sufficient to show that, for every $\gamma < 31/36$ (resp.\ $5/9$, $2/3$ and $1/9$) if $\star=1$ (resp.\ $1+$, $01$ and $010$), we have
\[
\sup_n \int_0^1 t^{-\gamma} \frac{\E \left[ f_n^\star(\omega(0)) f_n^\star(\omega(t)) \right]}{(\alpha^\star_n)^2} dt < +\infty.
\]
We first note that, by the results from \cite{LSW02,SW01} recalled in Section~\ref{ssec:exp}, $\gamma < 31/36$ (resp.\ $5/9$, $2/3$ and $1/9$) is equivalent to $1-\gamma>4\zeta^\star/3$.

Since $\E[ f_n^\star(\omega(0)) f_n^\star(\omega(t))] = \pi_n^\star(1 - e^{-t})$, it is sufficient to prove that
\[
\sup_n \int_0^1 t^{-\gamma} \frac{\pi^\star_n(t)}{(\alpha^\star_n)^2} dt  < +\infty
\]
for such $\gamma$'s. Recall that $\varepsilon_n = \frac{1}{n^2\alpha_n}$. For every $t \in [\varepsilon_n,1]$, we have $\ell(t) \leq \ell(\varepsilon_n)\leq n$. By using the bound $\pi^\star_n(t) \leq \alpha^\star_n$ on $[0,\varepsilon_n]$ and by using Theorem~\ref{thm:arm} and quasi-multiplicativity on $[\varepsilon_n,1]$ (as in \eqref{eq:etoile**}), we obtain that
\[
\int_0^1 t^{-\gamma} \frac{\pi^\star_n(t)}{(\alpha^\star_n)^2} dt \leq \int_0^{\varepsilon_n} t^{-\gamma} (\alpha^\star_n)^{-1} dt + C \int_{\varepsilon_n}^1 t^{-\gamma} (\alpha^\star_{\ell(t)})^{-1} dt.
\]
By the results from \cite{LSW02,SW01} recalled in Section~\ref{ssec:exp}, we have $\ell(t)=t^{-4/3+o(1)}$, $\varepsilon_n=n^{-3/4+o(1)}$ and $\alpha_m^\star=m^{-\zeta^\star+o(1)}$. As a result, the above equals
\begin{multline*}
\int_0^{n^{-3/4+o(1)}} t^{-\gamma}n^{\zeta^\star+o(1)} dt + C \int_{n^{-3/4+o(1)}}^1 t^{-\gamma} t^{4/3 \times \zeta^\star +o(1)} dt\\
\leq \frac{1}{1-\gamma} n^{-3(1-\gamma)/4+\zeta^\star+o(1)} + C\int_0^1 t^{-\gamma+4/3 \times \zeta^\star +o(1)} dt,
\end{multline*}
which is bounded uniformly in $n$ if $1-\gamma>4\zeta^\star/3$. This ends the proof.
\end{proof}

\section{Bond percolation on $\Z^2$}\label{sec:Z2}

In this section, we consider bond percolation on the square lattice $\Z^2$, which is defined as follows: Each edge is independently open with probability $p$ and closed with probability $1-p$. Consider the dual lattice $(\Z^2)^\star$ and let each edge $e^\star$ of this dual lattice be open (resp.\ closed) if the unique edge $e$ that intersects $e^\star$ is closed (resp.\ open). The primal open edges are colored in black and the dual open edges are colored in white. The critical parameter of this model is $p_c=1/2$ \cite{Kes80}.

In every proof except in Section~\ref{sec:dyna} (where we prove Theorem~\ref{thm:Hausdorff} by using the computation of the critical exponents for face percolation on the hexagonal lattice), we have only used results that are also known for bond percolation on $\Z^2$. In particular, we have proven Theorems~\ref{thm:GPS} and~\ref{thm:arm} for this model.

Garban, Pete and Schramm have proven in \cite{GPS10} that a.s.\ there exist exceptional times with an unbounded black component for percolation on $\Z^2$. If one combines their results with the bounds on the arm events probabilities from \cite{DMT20}, one obtains that the Hausdorff dimension of such times is at least $3/4$. By using Theorem~\ref{thm:arm} and the bounds from \cite{DMT20}, we obtain a new proof of this result and we also prove a lower bound on the Hausdorff dimension of exceptional times with both a black and a white unbounded components.
\begin{theorem}\label{thm:HausdorffZ2}
Consider critical dynamical bond percolation on $\Z^2$. There exist $d_1 \ge 3/4$ and $d_2 \ge 1/2$ such that a.s.\ the Hausdorff dimension of the set of times with an unbounded black component equals $d_1$ and the Hausdorff dimension of the set of times with both a black and a white unbounded components equals $d_2$.
\end{theorem}
\begin{proof}
The fact that these Hausdorff dimensions are a.s.\ constant is a consequence of Kolmogorov's 0-1 law as explained in \cite[Section 6]{SS10}. Let $\star=1$ or $01$. As in the proof of Theorem~\ref{thm:Hausdorff}, it is sufficient to show that for any $\gamma < 3/4$ (resp.\ $1/2$),
\[
\sup_n \varepsilon_n^{1-\gamma} (\alpha^\star_n)^{-1} + \int_0^1 t^{-\gamma} (\alpha^\star_{\ell(t)})^{-1} dt < +\infty.
\]
This is a direct consequence of the following estimates from \cite[Section 6.4]{DMT20}:
\[
\alpha_n^1 \ge cn^{-1/6}, \quad \alpha_n^{01} \ge cn^{-1/3}, \quad \alpha_n \ge cn^{-4/3}. \qedhere
\]
\end{proof}

\footnotesize

\bibliographystyle{alpha}
\bibliography{ref}

\end{document}